\pdfoutput=1
\documentclass[11pt]{article}

\usepackage{preprint-layout}
\usepackage{url}
\newcommand{\ppnotation}{}
\usepackage{preprint-notation}
\graphicspath{{figures/}}

\title{Normal Equations and Discrete Energy Structures for
High-Order Streamline Diffusion}
\author{
Erik Burman \mbox{  }
Peter Hansbo \mbox{  }
Mats G. Larson
}
\date{ }

\begin{document}

\maketitle

\begin{abstract}
We analyse fully discrete Streamline Diffusion finite element methods, also
known as streamline upwind/Petrov--Galerkin (SUPG) methods, for time-dependent
first-order systems with skew-symmetric spatial operators.
To the best of our knowledge, this is the first stability and a
priori error analysis of strongly consistent Streamline Diffusion finite
elements combined with linear multistep schemes of order greater than two for
multidimensional first-order hyperbolic systems.
Local-in-time \(L^2\)-residual minimisation singles out
\(\delta=b_0\tau\), where \(\delta\) is the stabilisation parameter,
\(\tau\) is the time step, and \(b_0\) is the current-time coefficient in the
method average. For this choice, each implicit linear system is a symmetric
positive definite normal equation in a graph norm; for a general
stabilisation choice, the resulting system need not be symmetric.

A direct dual-residual estimate, proved with the \(L^2\)-projection,
skew-symmetry, and discrete summation by parts, yields errors of order
\(O(h^{k+1/2}+\tau^\nu)\) when \(\delta\) is proportional to \(h\) and
\(\tau=O(h)\), under the stated regularity assumptions. Here \(h\) is the mesh
size, \(k\) is the spatial polynomial degree, and \(\nu\) is the temporal order. Stability is
established for the \(\theta\)-method and third- through fifth-order
Adams--Moulton methods (AM3--AM5). The general-degree AM3 and AM4 estimates use
strengthened Courant--Friedrichs--Lewy (CFL) conditions; an elementwise
cancellation recovers the standard hyperbolic regime, with \(\tau\) bounded by
a constant multiple of \(\delta\) and \(\delta\) by a constant multiple of
\(h\), for continuous piecewise affine elements.
AM5 and the explicit third- and fourth-order Adams--Bashforth methods
(AB3 and AB4) are stable under a standard hyperbolic CFL condition, though the
explicit schemes do not have normal-equation structure.

Order-matched acoustic tests recover orders two through five and
material-residual rates one half-order lower. In a d'Alembert test with
discontinuous initial data, normal-equation SUPG recovers second-order local
\(\mathbb P_1\)
convergence away from the wave fronts, whereas the unstabilised error is close
to half order. A two-dimensional compact-wave test shows improved localisation
with modest dissipation.
\end{abstract}


\section{Introduction}
\label{sec:intro}

\paragraph{Problem Setting.}
Let \(I=(0,T)\) be a finite time interval, let
\(\Omega\subset\mathbb R^d\) be a Lipschitz domain, and let
\(n_{\mathrm{sys}}\) denote the number of components. Set
\(\mathcal H=[L^2(\Omega)]^{n_{\mathrm{sys}}}\), with inner product
\((\,\cdot\,,\,\cdot\,)_{\Omega}\), and let
\(\mathcal V\subset[H^1(\Omega)]^{n_{\mathrm{sys}}}\) be a dense subspace
that incorporates the boundary conditions. We assume that the time-independent
first-order differential operator \(G:\mathcal V\to\mathcal H\) is
skew-symmetric on \(\mathcal V\), that is,
\begin{equation}
(GU,V)_{\Omega}=-(U,GV)_{\Omega}
\quad \forall U,V\in\mathcal V \label{eq:introduction-spatial-operator-skew}
\end{equation}
Given \(F:I\to\mathcal H\) and \(U_0\in\mathcal H\), we seek
\(U:[0,T]\to\mathcal H\), with \(U(t)\in\mathcal V\) for \(t\in I\), that
satisfies the Friedrichs-type initial-value problem
\begin{align}
\partial_tU+GU&=F \quad\text{in }I \label{eq:model}\\
U(0)&=U_0 \label{eq:model-initial-condition}
\end{align}
For the stationary theory of Friedrichs systems, we refer to
\cite{EG06,EGC07}; time-dependent formulations are considered in
\cite{BE16,BH21}.
In several space dimensions the operator acts on coupled fields, so
the analysis must control the interaction between multistep time differences
and the full system operator rather than a single scalar transport
direction.

Skew-symmetry endows the continuous problem with a natural \(L^2\)-energy
and yields energy stability for standard Galerkin finite element
semidiscretisations. For smooth solutions, however, the resulting error
estimates are typically suboptimal by one order. Stabilised methods commonly
recover half an order and, more importantly for wave propagation, better
preserve the local propagation of perturbations. The resulting discretisation
errors can therefore reflect more closely the dynamics of perturbations in the
continuous system, whereas dispersive errors generated by an unstabilised
Galerkin method may pollute spatially separated regions.

\paragraph{The Transient Stabilisation Problem.}
For a conforming finite element space \(\mathcal V_h\subset\mathcal V\), the
consistent semidiscrete Streamline Diffusion, or streamline
upwind/Petrov--Galerkin (SUPG), formulation takes the following form: find
\(U_h(t)\in\mathcal V_h\), with
\(U_h(0)=U_{h,0}\), where \(U_{h,0}\) is a suitable approximation of \(U_0\),
such that
\begin{equation}
\int_{\Omega}(\partial_tU_h+GU_h)\cdot(V_h+\delta GV_h)\,dx
=\int_{\Omega}F\cdot(V_h+\delta GV_h)\,dx
\quad \forall V_h\in\mathcal V_h \label{eq:introduction-supg-form}
\end{equation}
Here \(\delta>0\) is the stabilisation parameter, with the dimension of time.
Consistency requires the full
evolution residual to be paired with the perturbed test function. The
resulting cross term \((\partial_tU_h,\delta GV_h)_{\Omega}\) is the principal
obstacle in the transient energy analysis.

This analytical difficulty is coupled to the algebraic character of the time
step. For a general stabilisation parameter, an implicit Streamline Diffusion
step is nonsymmetric. Explicit Petrov--Galerkin schemes instead inherit a
generally nonsymmetric mass matrix and are subject to a hyperbolic
Courant--Friedrichs--Lewy (CFL) condition. This distinction is especially
relevant for nonlinear problems,
where implicit methods require a nonlinear solve at every time step.

In the framework developed below, a method average with current coefficient
\(b_0\) admits the distinguished parameter choice \(\delta=b_0\tau\). This
choice arises from local-in-time residual minimisation
\cite{BM04,FGK25} and makes the linear system at each time step a symmetric
positive definite normal equation in the corresponding graph norm. Other
choices of \(\delta\) remain useful for the stability analysis but do not in
general produce a symmetric system. For an explicit Adams--Bashforth method,
\(b_0=0\), so no positive stabilisation parameter yields this normal-equation
structure.

\paragraph{Main Contributions.}
We address these issues through a common operator and energy framework for
transient Streamline Diffusion methods combined with the \(\theta\)-method
and high-order Adams schemes.
Beyond order two, the classical \(A\)-stability route for linear
multistep methods is unavailable \cite{Dahl63}; stability must instead be
recovered from method-dependent discrete energies that remain compatible with
the spatial residual.
The principal contributions are as follows. Here and below, \(h\) denotes the
mesh size, \(k\) the spatial polynomial degree, \(\tau\) the time step, and
\(\nu\) the temporal order.
\begin{enumerate}
\item The fully discrete formulation separates the method average, the
backward time difference, the material residual, and the perturbed test
operator. This organisation exposes a common algebraic identity and
distinguishes the residual-minimisation choice \(\delta=b_0\tau\) from a
general SUPG parameter.

\item A direct dual-residual estimate leads to a fully discrete finite element
error bound of order \(O(h^{k+1/2}+\tau^\nu)\) when \(\delta\sim h\) and
\(\tau\lesssim h\), under the stated regularity assumptions. The proof uses
the ordinary \(L^2\)-projection, skew-symmetry, and discrete summation by
parts.

\item The abstract energy hypothesis is verified for the \(\theta\)-method and
for the Adams--Moulton (AM) methods AM3--AM5. For general polynomial degree,
the AM3 and AM4 results require strengthened CFL conditions. For continuous
piecewise affine elements, an argument based on elementwise constant gradients
recovers stability
and the error estimate under the standard hyperbolic CFL condition when
\(\tau\lesssim\delta\lesssim h\). AM5 requires only the standard hyperbolic
CFL condition for every degree covered by the abstract framework. Independent
energy identities establish stability and error estimates for the
Adams--Bashforth (AB) methods AB3 and AB4 under a standard hyperbolic CFL
condition.

\item The experiments separate temporal accuracy from spatial approximation.
An order-matched one-dimensional acoustic test pairs a method of temporal
order \(\nu\) with \(\mathbb P_{\nu-1}\) elements and recovers orders two through five in
the fully discrete final and graph errors; the material-residual errors follow
the corresponding half-order sequence. A discontinuous solution constructed
by d'Alembert's formula shows second-order local \(\mathbb P_1\) convergence for
normal-equation SUPG in smooth regions separated from the wave fronts, whereas
the unstabilised local error converges with order close to \(1/2\). A smooth
two-dimensional compact-wave test provides complementary evidence of improved
localisation.
\end{enumerate}

\noindent\begin{minipage}{\textwidth}
\paragraph{Previous Work.}
Hughes and Brooks introduced the Streamline Upwind Petrov--Galerkin method for
stationary hyperbolic problems in \cite{BrooksHughes1982CMAME}. Johnson and
co-workers subsequently gave a comprehensive analysis for stationary linear
transport in \cite{JNP84}, which also discusses the application of Streamline
Diffusion to Friedrichs systems.
\end{minipage}

The extension to transient problems is more delicate because consistency
requires the full space--time residual to be tested against the perturbed test
function. The resulting cross term couples the discrete time derivative to the
spatial differential operator. Bochev and co-workers established invertibility
of the linear system in \cite{BGS04}, but their result did not provide uniform
control of error growth in time. For finite-difference time discretisations,
tuning the stabilisation parameter was proposed as a means of controlling the
cross term \cite{LW95}. Space--time finite element methods with discontinuous
Galerkin approximation in time offered another route \cite{JNP84,HFH89}, at
the cost of larger, fully coupled nonsymmetric systems.
Higher-order space--time Streamline Diffusion analyses are available
for specialised kinetic systems \cite{AKS19}, while an upwind discontinuous
Galerkin time discretisation with a space--time SUPG term has recently been
analysed for scalar advection--diffusion \cite{BDG25}. These methods discretise
time variationally and do not produce the separated linear multistep
formulation considered here.

Discontinuous Galerkin methods combined with Runge--Kutta time integration
were developed in \cite{CS89} and became widely used for hyperbolic problems.
For continuous finite elements, symmetric alternatives include subgrid
viscosity \cite{Guer99}, orthogonal subscales \cite{Cod00}, local projection
\cite{BB04}, and continuous interior penalty stabilisation \cite{BH04}.
When combined with explicit Runge--Kutta methods, these approaches reproduce
important stability properties of discontinuous Galerkin schemes
\cite{BEF10}, although the stabilisation generally enlarges the spatial matrix
stencil.
High-order continuous finite elements with SUPG and Runge--Kutta or
deferred-correction time integration have also been studied by fully discrete
Fourier analysis on two-dimensional triangular meshes \cite{MTRA23}. That
analysis characterises linear stability and admissible CFL parameters, rather
than deriving the system-level a priori energy and error estimates pursued
here.

Uniform stability of the transient SUPG method for linear transport was
established by Burman \cite{Bu10}. Testing with a linear combination of the
solution and its time derivative yielded a stability estimate for any
\(A\)-stable time discretisation, including Crank--Nicolson and the
second-order backward differentiation formula. The analysis was subsequently extended to
convection--diffusion problems \cite{BS11} and time-dependent coefficients
\cite{ZL22}. More recently, stability has been proved for third- and
fourth-order Runge--Kutta methods combined with Streamline Diffusion in space
\cite{EG25}.
Thus the closest rigorous results either stop at second-order linear
multistep schemes or use Runge--Kutta, discontinuous-in-time, or coupled
space--time formulations. To the best of our knowledge, the present work
provides the first stability and a priori error analysis of strongly
consistent Streamline Diffusion finite elements combined with linear
multistep schemes of order greater than two for multidimensional first-order
hyperbolic systems. This is not merely a change of time integrator: the loss
of \(A\)-stability beyond order two leads to the method-dependent discrete
energy mechanisms developed here for AM3--AM5 and AB3--AB4.

\paragraph{Outline.}
Section~\ref{sec:abstract-method} introduces the operator setting, the time
average and its adjoint, and the operators \(T_{\delta}\), \(S_{\delta}\),
\(A_{\tau}\), and \(P_{\delta}\). Section~\ref{sec:meta} develops the
abstract stability and error framework. Section~\ref{sec:theta-method} proves
the \(\theta\)-method energy identity; Section~\ref{sec:adams-moulton} analyses AM3--AM5;
and Section~\ref{sec:ab} derives the AB3 and AB4 energy and stability
estimates. Section~\ref{sec:numerical-experiments} presents the one- and
two-dimensional acoustic experiments, and Section~\ref{sec:conclusions}
summarises the results and open problems. Appendix~\ref{app:dual-residual-proof}
contains the horizon-uniform proof of the dual-residual estimate.


\section{Abstract Setting and Formulation of the Method}
\label{sec:abstract-method}
We formulate the method for a real Hilbert space \(\mathcal H\) with inner
product \((\,\cdot\,,\,\cdot\,)_{\Omega}\) and a dense subspace
\(\mathcal V\subset\mathcal H\). The notation introduced here is used in all
subsequent sections. The relations \(a\lesssim b\) and \(a\sim b\) mean,
respectively, that \(a\le Cb\) and \(c b\le a\le Cb\), with positive
constants \(c\) and \(C\) independent of the mesh size, time step, and final
computed index unless stated otherwise.

\paragraph{Spatial Operator and Skew-Symmetry.}
Let \(G\) be a linear operator satisfying
\begin{equation}
G:\mathcal V\longrightarrow\mathcal H \label{eq:spatial-operator-domain}
\end{equation}
and assume that the boundary conditions and the choice of \(\mathcal V\)
give
\begin{equation}
(GU,V)_{\Omega}=-(U,GV)_{\Omega}\quad \forall U,V\in\mathcal V \label{eq:spatial-operator-skew}
\end{equation}
Consequently,
\begin{equation}
(GU,U)_{\Omega}=0\quad \forall U\in\mathcal V \label{eq:spatial-operator-null}
\end{equation}
The analysis below is for a time-independent operator. Thus, every time
difference and time average commutes with \(G\). Space- or time-dependent
coefficients require additional commutator bounds and are not covered by the
present statements.

\paragraph{Example: Acoustic Waves.}
Let \(c>0\) be constant. Starting from the physical variables \(p\) and
\(v\), introduce the scaled velocity \(u=cv\). The acoustic system becomes
\begin{align}
\partial_t p+c\nabla\cdot u&=f &&\text{in }\Omega\times(0,T) \label{eq:acoustic-pressure}\\
\partial_t u+c\nabla p&=cg &&\text{in }\Omega\times(0,T) \label{eq:acoustic-velocity}
\end{align}
Set \(U=(p,u)\), \(F=(f,cg)\), and
\begin{equation}
\mathcal H=L^2(\Omega)\times L^2(\Omega)^d \label{eq:acoustic-hilbert-space}
\end{equation}
with the product inner product. The operator is
\begin{equation}
G(p,u)=\bigl(c\nabla\cdot u,c\nabla p\bigr) \label{eq:acoustic-spatial-operator}
\end{equation}
For periodic boundary conditions, or reflecting-wall boundary conditions
for which the boundary pairing vanishes, integration by parts proves
\eqref{eq:spatial-operator-skew}. The corresponding norm is
\begin{equation}
\|U\|_{\Omega}^2=\|p\|_{\Omega}^2+\|u\|_{\Omega}^2
=\|p\|_{\Omega}^2+c^2\|v\|_{\Omega}^2 \label{eq:acoustic-energy-norm}
\end{equation}
Thus the powers of \(c\) in \eqref{eq:acoustic-pressure}--\eqref{eq:acoustic-velocity},
the operator, and the energy are consistent.


\subsection{Time Differences, Averages, and Their Adjoints}
\label{subsec:time-operators}
Let \(N\) denote the final computed index, let \(\tau>0\), and set
\(t^n=n\tau\) for \(0\le n\le N\), with \(t^N\le T\). Define the backward and
forward time differences by
\begin{align}
D_{\tau}^{-}U^n&=\tau^{-1}(U^n-U^{n-1}) \label{eq:backward-time-difference}\\
D_{\tau}^{+}U^n&=\tau^{-1}(U^{n+1}-U^n) \label{eq:forward-time-difference}
\end{align}
For coefficients \(b_0,\ldots,b_{\ell}\), define the method average
\begin{equation}
\mathcal M_{\tau}U^n=\sum_{j=0}^{\ell}b_jU^{n-j} \label{eq:method-average}
\end{equation}
Thus \(\ell+1\) is the stencil length, and \(b_0\) is the coefficient of the
current value.
The same notation is used for sampled exact functions and for the data. The
first computed index is denoted by \(n_0\), with \(n_0\ge\ell\), and all
earlier values required by the stencil are prescribed start-up data.

The average \(\mathcal M_{\tau}\) is not self-adjoint on a finite time
interval. Extend a test sequence by zero outside \(n_0\le n\le N\), and define
its discrete adjoint average by
\begin{equation}
\mathcal M_{\tau}^{*}V^m=\sum_{j=0}^{\ell}b_jV^{m+j} \label{eq:adjoint-method-average}
\end{equation}
A change of summation index gives the exact identity
\begin{equation}
\sum_{n=n_0}^{N}(\mathcal M_{\tau}U^n,V^n)_{\Omega}
=\sum_{m=n_0-\ell}^{N}(U^m,\mathcal M_{\tau}^{*}V^m)_{\Omega} \label{eq:average-adjoint-identity}
\end{equation}
In particular, \(\mathcal M_{\tau}^{*}V^m\) cannot be replaced by
\(\mathcal M_{\tau}V^m\) without additional symmetry and endpoint terms.
The corresponding summation-by-parts identity is
\begin{align}
\tau\sum_{n=n_0}^{N}(D_{\tau}^{-}U^n,V^n)_{\Omega}
&=(U^N,V^N)_{\Omega}-(U^{n_0-1},V^{n_0})_{\Omega} \label{eq:time-summation-by-parts-endpoints}\\
&\quad-\tau\sum_{n=n_0}^{N-1}(U^n,D_{\tau}^{+}V^n)_{\Omega} \label{eq:time-summation-by-parts-bulk}
\end{align}
Both endpoint terms in \eqref{eq:time-summation-by-parts-endpoints} are kept
explicit whenever this formula is used.

\paragraph{Approximation and Local Truncation Error.}
Let \(\mathcal T_h\) be a shape-regular mesh of \(\Omega\), set
\(h=\max_{K\in\mathcal T_h}\operatorname{diam}(K)\), and let
\(\mathbb P_j(K)\) denote the scalar polynomials of total degree at most
\(j\) on \(K\). Let \(\mathcal V_h\subset\mathcal V\) be a conforming finite
element space of componentwise polynomial degree at most \(k\), and let
\(\Pi_h:\mathcal H\to\mathcal V_h\) be the \(L^2\)-orthogonal projection. If
\(s\ge1\) is the spatial Sobolev regularity index, set
\(r=\min(s,k+1)\). We assume that
\begin{equation}
\|U-\Pi_hU\|_{\Omega}+h\|G(U-\Pi_hU)\|_{\Omega}
\lesssim h^r\|U\|_{H^r(\Omega)} \label{eq:spatial-approximation}
\end{equation}
Because \(\Pi_h\) is time independent, it commutes with the operators in
\eqref{eq:backward-time-difference} and \eqref{eq:method-average}. Define the
local truncation residual by
\begin{equation}
\rho^n=D_{\tau}^{-}U(t^n)+G\mathcal M_{\tau}U(t^n)
-\mathcal M_{\tau}F(t^n) \label{eq:local-truncation-residual}
\end{equation}
For a method of temporal order \(\nu\), we assume the standard local estimate
\begin{equation}
\|\rho^n\|_{\Omega}\le C\tau^{\nu-1/2}\mathcal C_{n,\nu+1}(U,F) \label{eq:high-consist}
\end{equation}
where \(\mathcal C_{n,\nu+1}(U,F)\) is an \(L^2\)-in-time regularity norm on the
stencil interval. The Adams estimate is derived in Section
\ref{sec:trunc-adams}.
Define the corresponding global regularity quantity by
\begin{equation}
\mathcal C_{T,\nu+1}(U,F)
=\left(\sum_{n=n_0}^{N}\mathcal C_{n,\nu+1}(U,F)^2\right)^{1/2}
\label{eq:global-truncation-regularity}
\end{equation}
The stencil intervals have uniformly bounded overlap (with a bound depending
only on the fixed method), so this quantity is controlled by the appropriate
global \(L^2(0,T)\) norm of the time derivatives occurring in the local
remainder. In particular, \eqref{eq:high-consist} gives
\begin{equation}
\|\rho\|_{\ell_{\tau}^2(n_0,N;\mathcal H)}
\le C\tau^\nu\mathcal C_{T,\nu+1}(U,F)
\label{eq:global-truncation-bound}
\end{equation}
For the dual-residual estimate we use the corresponding differentiated
consistency bounds
\begin{align}
\max_{n_0\le n\le N}\|\rho^n\|_{\Omega}
&\lesssim\tau^\nu\mathcal C^{\sharp}_{T,\nu+2}(U,F)
\label{eq:maximum-truncation-bound}\\
\|D_{\tau}^{+}\rho\|_{\ell_{\tau}^2(n_0,N-1;\mathcal H)}
&\lesssim\tau^\nu\mathcal C^{\sharp}_{T,\nu+2}(U,F)
\label{eq:differentiated-truncation-bound}
\end{align}
Here \(\mathcal C^{\sharp}_{T,\nu+2}(U,F)\) contains one additional time
derivative compared with \(\mathcal C_{T,\nu+1}(U,F)\), together with the
endpoint regularity required for the maximum in
\eqref{eq:maximum-truncation-bound}. These estimates follow from the same
integral-remainder argument as \eqref{eq:high-consist}, applied also to the
first forward difference of the remainder. They are verified below for the
\(\theta\)- and Adams methods.


\subsection{Shift and Method Operators}
\label{subsec:shift-method-operators}
For a stabilisation parameter \(\delta>0\), define the spatial test operator,
time shift, method operator, and combined test operator by
\begin{align}
T_{\delta}&=\mathrm{Id}+\delta G \label{eq:spatial-test-operator}\\
S_{\delta}&=\mathcal M_{\tau}+\delta D_{\tau}^{-} \label{eq:time-shift-operator}\\
A_{\tau}&=D_{\tau}^{-}+G\mathcal M_{\tau} \label{eq:method-operator}\\
P_{\delta}&=T_{\delta}S_{\delta} \label{eq:combined-test-operator}
\end{align}
Commutation with \(G\) gives
\begin{equation}
P_{\delta}=\mathcal M_{\tau}+\delta D_{\tau}^{-}
+\delta G\mathcal M_{\tau}+\delta^2GD_{\tau}^{-} \label{eq:combined-test-expansion}
\end{equation}
Skew-symmetry gives the graph-inner-product identity
\begin{equation}
(T_{\delta}U,T_{\delta}V)_{\Omega}
=(U,V)_{\Omega}+\delta^2(GU,GV)_{\Omega} \label{eq:test-operator-inner-product}
\end{equation}
and hence
\begin{equation}
\|T_{\delta}U\|_{\Omega}^2
=\|U\|_{\Omega}^2+\delta^2\|GU\|_{\Omega}^2 \label{eq:test-operator-norm}
\end{equation}
We use the graph inner product and norm
\begin{align}
(U,V)_{\delta}
&=(U,V)_{\Omega}+\delta^2(GU,GV)_{\Omega} \label{eq:graph-inner-product}\\
\|U\|_{\delta}^2
&=(U,U)_{\delta}=\|T_{\delta}U\|_{\Omega}^2 \label{eq:graph-norm}
\end{align}


\subsection{SUPG and Residual-Minimisation Formulations}
\label{subsec:supg-residual-formulations}
The time-discrete SUPG method is: given the start-up values, find
\(U^n\in\mathcal V_h\) such that
\begin{equation}
(A_{\tau}U^n,T_{\delta}V)_{\Omega}
=(\mathcal M_{\tau}F^n,T_{\delta}V)_{\Omega}
\quad \forall V\in\mathcal V_h \label{eq:supg-weak}
\end{equation}
This definition allows \(\delta\) to be chosen independently of the method
coefficient \(b_0\).

\paragraph{Residual-Minimisation Case.}
At a fixed time level, consider
\begin{equation}
U^n=\operatorname*{argmin}_{V^n\in\mathcal V_h}
\|D_{\tau}^{-}V^n+G\mathcal M_{\tau}V^n
-\mathcal M_{\tau}F^n\|_{\Omega}^2 \label{eq:ls}
\end{equation}
The derivative of the residual with respect to its current value is
\(\tau^{-1}T_{b_0\tau}\). Therefore the Euler--Lagrange equation is
\begin{equation}
(A_{\tau}U^n,T_{b_0\tau}V)_{\Omega}
=(\mathcal M_{\tau}F^n,T_{b_0\tau}V)_{\Omega}
\quad \forall V\in\mathcal V_h \label{eq:normal-equation}
\end{equation}
Thus \eqref{eq:supg-weak} is a normal equation precisely when
\(\delta=b_0\tau\).

Define the history part of the average by
\begin{equation}
\mathcal M_{\tau}^{-}U^n=\sum_{j=1}^{\ell}b_jU^{n-j} \label{eq:history-average}
\end{equation}
Separating the current value in \eqref{eq:normal-equation} gives the update
equation with the correct signs,
\begin{align}
(T_{b_0\tau}U^n,T_{b_0\tau}V)_{\Omega}
&=(U^{n-1}-\tau G\mathcal M_{\tau}^{-}U^n
+\tau\mathcal M_{\tau}F^n,T_{b_0\tau}V)_{\Omega} \label{eq:normal-update-equation}
\end{align}
The left-hand side is symmetric positive definite in the graph norm
\eqref{eq:test-operator-norm}. For a general stabilisation parameter
\(\delta\ne b_0\tau\), the SUPG system need not be symmetric.


\subsection{Acoustic-Wave Theta Method}
\label{subsec:acoustic-wave-equation}
For \(1/2\le\theta\le1\), set
\begin{equation}
\mathcal M_{\theta}U^n=\theta U^n+(1-\theta)U^{n-1} \label{eq:theta-average-acoustic}
\end{equation}
This is the \(\theta\)-method specialisation of the general average
\(\mathcal M_{\tau}\).
For \(U_h^n=(p_h^n,u_h^n)\), define the two current residual components by
\begin{align}
R_p^n&=D_{\tau}^{-}p_h^n+c\nabla\cdot\mathcal M_{\theta}u_h^n
-\mathcal M_{\theta}f^n \label{eq:acoustic-pressure-residual}\\
R_u^n&=D_{\tau}^{-}u_h^n+c\nabla\mathcal M_{\theta}p_h^n
-c\mathcal M_{\theta}g^n \label{eq:acoustic-velocity-residual}
\end{align}
The component form of \eqref{eq:supg-weak} is
\begin{equation}
(R_p^n,q_h+\delta c\nabla\cdot w_h)_{\Omega}
+(R_u^n,w_h+\delta c\nabla q_h)_{\Omega}=0
\quad \forall(q_h,w_h)\in\mathcal V_h \label{eq:acoustic-theta-variational}
\end{equation}
All terms in \eqref{eq:acoustic-theta-variational} are evaluated at the
current residual level \(n\); the previous level occurs only through
\(D_{\tau}^{-}\) and \(\mathcal M_{\theta}\).

When \(\delta=\theta\tau\), define
\begin{equation}
H^{n-1}=U_h^{n-1}-\tau(1-\theta)GU_h^{n-1}
+\tau\mathcal M_{\theta}F^n \label{eq:acoustic-theta-history}
\end{equation}
Equation \eqref{eq:normal-update-equation} becomes
\begin{align}
(p_h^n,q_h)_{\Omega}+(u_h^n,w_h)_{\Omega}
&+\delta^2c^2(\nabla p_h^n,\nabla q_h)_{\Omega} \label{eq:acoustic-theta-spd-pressure}\\
&+\delta^2c^2(\nabla\cdot u_h^n,\nabla\cdot w_h)_{\Omega}
=(H^{n-1},T_{\delta}(q_h,w_h))_{\Omega} \label{eq:acoustic-theta-spd-velocity}
\end{align}
The current pressure and scaled velocity are uncoupled on the left-hand side,
and the two diagonal blocks are symmetric positive definite. This conclusion
uses both the rescaling \(u=cv\) and the normal-equation choice
\(\delta=\theta\tau\).


\section{Abstract Stability and Error Analysis}
\label{sec:meta}
This section turns the common algebraic identity into a stability and error
framework. We first introduce sequence norms that accommodate the different
method-dependent energies. We then formulate an abstract energy hypothesis,
derive the error equation, and estimate its residual in both a strong norm and
a lifted dual norm. The final parameter balance makes the resulting spatial
order explicit. The later method-specific sections verify the energy
hypothesis by exact one-step identities or by discrete summation by parts.

\paragraph{Fundamental Algebraic Identity.}
In \eqref{eq:supg-weak}, choose \(W=S_{\delta}U^n\), which is admissible because
\(S_{\delta}U^n\in\mathcal V_h\). Since
\(P_{\delta}=T_{\delta}S_{\delta}\), multiplication by \(\tau\) gives
\begin{equation}
(\tau A_{\tau}U^n,P_{\delta}U^n)_{\Omega}
=(\tau\mathcal M_{\tau}F^n,P_{\delta}U^n)_{\Omega} \label{eq:abstract-tested-method}
\end{equation}
Using \eqref{eq:combined-test-expansion}, skew-symmetry, and commutation of
\(G\) with the time operators gives
\begin{align}
(\tau A_{\tau}U^n,P_{\delta}U^n)_{\Omega}
&=\tau(D_{\tau}^{-}U^n,\mathcal M_{\tau}U^n)_{\Omega}
+\tau\delta\|A_{\tau}U^n\|_{\Omega}^2 \label{eq:stability_meta}\\
&\quad+\tau\delta^2(G\mathcal M_{\tau}U^n,
GD_{\tau}^{-}U^n)_{\Omega} \label{eq:abstract-tested-expansion-graph}
\end{align}
In particular, the material-residual contribution is
\(\tau\delta\|A_{\tau}U^n\|_{\Omega}^2\). This fixes the energy and
dissipation scaling used below.


\subsection{Sequence Norms and Abstract Stability}
\label{subsec:method-adapted-norms}
The energy estimate must control both the solution at every computed time and
the accumulated dissipation. Their precise forms depend on the time-stepping
method, so we encode them abstractly.
Let \(e_{\tau}^n(V)\ge0\) be a method-dependent energy and let
\(d_{\tau}^n(V)\ge0\) be its dissipation. For a horizon
\(n_0\le m\le N\), define
\begin{equation}
\|V\|_{\mathcal E,m}^2
=\max_{n_0\le j\le m}e_{\tau}^j(V)
+\tau\sum_{n=n_0}^{m}d_{\tau}^n(V) \label{eq:method-adapted-sequence-norm}
\end{equation}
Thus the sequence norm contains only computed levels. The prescribed levels
are recorded separately in the start-up energy. Specifically, the start-up
energy and the discrete \(L^2\)-in-time norm are
\begin{align}
\mathcal I_{\tau}(V)&=\sum_{j=0}^{n_0-1}\|T_{\delta}V^j\|_{\Omega}^2 \label{eq:startup-energy}\\
\|R\|_{\ell_{\tau}^2(n_0,m;\mathcal H)}^2
&=\tau\sum_{n=n_0}^{m}\|R^n\|_{\Omega}^2 \label{eq:discrete-l2-time-norm}
\end{align}
The forcing term in the energy identity is naturally paired with
\(P_{\delta}V^n\). We therefore define the corresponding global lifted dual
norm by
\begin{equation}
\|R\|_{\mathcal E^*,N}
=\max_{n_0\le m\le N}
\sup_{\substack{V^n=0\text{ for }n\notin\{0,\ldots,m\}\\
\mathcal I_{\tau}(V)+\|V\|_{\mathcal E,m}^2>0}}
\frac{\tau\sum_{n=n_0}^{m}(R^n,P_{\delta}V^n)_{\Omega}}
{\bigl(\mathcal I_{\tau}(V)+\|V\|_{\mathcal E,m}^2\bigr)^{1/2}}
\label{eq:lifted-dual-norm}
\end{equation}
The maximum over horizons is needed because the sequence norm contains a
maximum in time. The finite start-up stencil is retained and controlled by
\(\mathcal I_{\tau}(V)\), and the test sequence is extended by zero outside
\(0\le n\le m\) before \(P_{\delta}\) is applied. The positive-denominator
condition removes sequences invisible to the horizon norm. The factor \(\tau\)
matches the tested equation \eqref{eq:abstract-tested-method}.

\paragraph{Energy Hypothesis.}
The hypothesis has two parts: coercivity of the method-dependent energy and
dissipation, and a forced energy estimate uniform in the terminal horizon. We
assume that there are constants \(c_E,C_E>0\), independent of \(h\), \(\tau\),
and \(m\), such that
\begin{align}
c_E\|T_{\delta}V^m\|_{\Omega}^2
&\le e_{\tau}^m(V) \label{eq:energy-coercivity}\\
c_E\delta\|A_{\tau}V^n\|_{\Omega}^2
&\le d_{\tau}^n(V) \label{eq:dissipation-coercivity}
\end{align}
and every solution of
\begin{equation}
(A_{\tau}V^n,T_{\delta}W)_{\Omega}
=(R^n,T_{\delta}W)_{\Omega}
\quad \forall W\in\mathcal V_h \label{eq:abstract-forced-equation}
\end{equation}
satisfies, for every \(m\),
\begin{align}
e_{\tau}^m(V)+\tau\sum_{n=n_0}^{m}d_{\tau}^n(V)
&\le C_E\mathcal I_{\tau}(V)
+C_E\tau\sum_{n=n_0}^{m}(R^n,P_{\delta}V^n)_{\Omega} \label{eq:abstract-energy-hypothesis-main}\\
&\quad+C_E\tau\|R\|_{\ell_{\tau}^2(n_0,m;\mathcal H)}^2 \label{eq:abstract-energy-hypothesis-defect}
\end{align}
The right-hand side separates start-up data, the forcing functional measured
by the lifted dual norm above, and a strong forcing defect. The last term
permits the endpoint defects produced when increments are estimated in a
multistep method; it is absent from the \(\theta\)-method identity.

\begin{thm}[Abstract stability]
\label{thm:meta-stability}\label{thm:stability-meta}
Assume \eqref{eq:energy-coercivity}--\eqref{eq:abstract-energy-hypothesis-defect}.
Then
\begin{align}
\|V\|_{\mathcal E,N}^2
&\lesssim\mathcal I_{\tau}(V)
+\|R\|_{\mathcal E^*,N}^2
+\tau\|R\|_{\ell_{\tau}^2(n_0,N;\mathcal H)}^2 \label{eq:abstract-stability-estimate}
\end{align}
where the hidden constant is independent of \(h\) and \(\tau\).
\end{thm}
\begin{proof}
Choose \(m_*\), \(n_0\le m_*\le N\), at which the computed-level energy is
maximal. The energy inequality with \(m=m_*\) controls the maximum-energy
term in \eqref{eq:method-adapted-sequence-norm}, while the inequality with
\(m=N\) controls the full dissipation sum. By
\eqref{eq:lifted-dual-norm}, the forcing contribution is bounded by
\begin{equation}
\tau\sum_{n=n_0}^{m}(R^n,P_{\delta}V^n)_{\Omega}
\le\|R\|_{\mathcal E^*,N}
\bigl(\mathcal I_{\tau}(V)+\|V\|_{\mathcal E,m}^2\bigr)^{1/2}
\label{eq:dual-forcing-bound}
\end{equation}
Add the two horizon inequalities. Their left-hand sides control
\(\|V\|_{\mathcal E,N}^2\) up to a fixed constant. After bounding each horizon
norm by \(\|V\|_{\mathcal E,N}\), Young's inequality with a sufficiently small
parameter absorbs the resulting contribution to
\(\|V\|_{\mathcal E,N}^2\). The accompanying start-up factor is bounded by
\(C\mathcal I_{\tau}(V)\). This proves
\eqref{eq:abstract-stability-estimate}.
\end{proof}


\subsection{Error Equation and Residual Bounds}
\label{subsec:residual-error-equation}
We split the total error into an approximation error \(\eta^n\), controlled by
the projection estimates, and a discrete error \(\phi^n\), to which the
abstract stability theorem will be applied. Set
\begin{equation}
W^n=\Pi_hU(t^n),\qquad
\eta^n=U(t^n)-W^n,\qquad
\phi^n=U_h^n-W^n \label{eq:error-splitting}
\end{equation}
By \eqref{eq:local-truncation-residual}, the sampled exact solution satisfies
\begin{equation}
A_{\tau}U(t^n)=\mathcal M_{\tau}F(t^n)+\rho^n \label{eq:sampled-exact-equation}
\end{equation}
Subtracting the projected exact equation from the numerical method gives
\begin{equation}
(A_{\tau}\phi^n,T_{\delta}V)_{\Omega}
=(R^n,T_{\delta}V)_{\Omega}
\quad \forall V\in\mathcal V_h \label{eq:error-equation}
\end{equation}
with the residual
\begin{align}
R^n&=\mathcal M_{\tau}F^n-A_{\tau}W^n \label{eq:discrete-error-residual-definition}\\
&=A_{\tau}\eta^n-\rho^n \label{eq:discrete-error-residual-splitting}
\end{align}
The sign in \eqref{eq:discrete-error-residual-splitting} follows from
\eqref{eq:sampled-exact-equation}; in particular, the truncation residual
enters with a minus sign.

\paragraph{Strong Residual Bound.}
The abstract stability estimate contains the strong residual norm multiplied
by \(\tau^{1/2}\). The following bound controls that term; the sharper spatial
rate will come from the dual estimate below.
For \(v\in H^1(0,T;\mathcal H)\), the endpoint trace inequality on each
interval \((t^{n-1},t^n)\), followed by summation, gives
\begin{equation}
\tau\sum_{n=n_0}^{N}\|v(t^n)\|_{\Omega}^2
\lesssim \|v\|_{L^2(0,T;\mathcal H)}^2
+\tau^2\|\partial_tv\|_{L^2(0,T;\mathcal H)}^2
\label{eq:discrete-time-trace}
\end{equation}
For each fixed shift \(j\) in the finite stencil, changing the summation
index in \eqref{eq:discrete-time-trace} gives
\begin{equation}
\tau\sum_{n=n_0}^{N}\|v(t^{n-j})\|_{\Omega}^2
\lesssim \|v\|_{L^2(0,T;\mathcal H)}^2
+\tau^2\|\partial_tv\|_{L^2(0,T;\mathcal H)}^2
\label{eq:shifted-discrete-time-trace}
\end{equation}
Here \(n_0\ge\ell\), so every sampled time belongs to \([0,T]\). Since the
number of shifts and the method coefficients are fixed, it follows that
\begin{equation}
\|\mathcal M_{\tau}v\|_{\ell_{\tau}^2(n_0,N;\mathcal H)}
\lesssim \|v\|_{L^2(0,T;\mathcal H)}
+\tau\|\partial_tv\|_{L^2(0,T;\mathcal H)}
\label{eq:method-average-discrete-time-trace}
\end{equation}
Apply \eqref{eq:method-average-discrete-time-trace} to
\(v=G\eta=G(\mathrm{Id}-\Pi_h)U\), which belongs to
\(H^1(0,T;\mathcal H)\) under the regularity used below. The time
independence of \(G\) and \(\Pi_h\), together with
\eqref{eq:spatial-approximation} applied to \(U\) and \(\partial_tU\), gives
\begin{align}
\|G\mathcal M_{\tau}\eta\|_{\ell_{\tau}^2(n_0,N;\mathcal H)}
&\lesssim h^{r-1}\|U\|_{L^2(0,T;H^r(\Omega))} \label{eq:strong-residual-space-part}\\
&\quad+\tau h^{r-1}
\|\partial_tU\|_{L^2(0,T;H^r(\Omega))}
\label{eq:strong-residual-space-trace-part}
\end{align}
On the other hand, writing \(D_{\tau}^{-}\eta^n\) as the average of
\((\mathrm{Id}-\Pi_h)\partial_tU\) over \((t^{n-1},t^n)\), and then using the
Cauchy--Schwarz inequality, gives
\begin{equation}
\|D_{\tau}^{-}\eta\|_{\ell_{\tau}^2(n_0,N;\mathcal H)}
\lesssim h^r\|\partial_tU\|_{L^2(0,T;H^r(\Omega))}
\label{eq:strong-residual-time-part}
\end{equation}
Consequently,
\begin{align}
\|A_{\tau}\eta\|_{\ell_{\tau}^2(n_0,N;\mathcal H)}
&\lesssim h^{r-1}\|U\|_{L^2(0,T;H^r(\Omega))}
\label{eq:strong-projection-residual-space}\\
&\quad+\bigl(h^r+\tau h^{r-1}\bigr)
\|\partial_tU\|_{L^2(0,T;H^r(\Omega))}
\label{eq:strong-projection-residual-time}
\end{align}
Together with \eqref{eq:global-truncation-bound}, this proves
\begin{align}
\|R\|_{\ell_{\tau}^2(n_0,N;\mathcal H)}
&\lesssim h^{r-1}\|U\|_{L^2(0,T;H^r(\Omega))}
+\bigl(h^r+\tau h^{r-1}\bigr)
\|\partial_tU\|_{L^2(0,T;H^r(\Omega))}
\label{eq:strong-residual-bound-space}\\
&\quad+\tau^\nu\mathcal C_{T,\nu+1}(U,F) \label{eq:strong-residual-bound-time}
\end{align}
No factor \(\tau^{-1}h^r\) occurs because the time difference of the
projection error is estimated as a difference quotient of a smooth function.

\paragraph{Dual Residual Bound.}
The gain of one half order is obtained by estimating the residual functional
itself, rather than its strong norm. Orthogonality of the \(L^2\)-projection
allows the functional to be reorganised in terms of precisely the quantities
controlled by the energy and dissipation. Define
\begin{align}
\mathcal N_{r,2}(U)
&=\sum_{j=0}^{2}\|\partial_t^jU\|_{L^2(0,T;H^r(\Omega))}
\label{eq:dual-regularity-l2}\\
&\quad+\sum_{j=0}^{1}\|\partial_t^jU\|_{L^{\infty}(0,T;H^r(\Omega))}
\label{eq:dual-regularity-uniform}
\end{align}

\begin{prop}[Dual residual estimate]
\label{prop:dual-residual-bound}
Assume \eqref{eq:energy-coercivity}--\eqref{eq:dissipation-coercivity},
\(0<\delta\le1\), the approximation estimate
\eqref{eq:spatial-approximation} for \(U\), \(\partial_tU\), and
\(\partial_t^2U\), and the differentiated consistency estimates
\eqref{eq:maximum-truncation-bound}--\eqref{eq:differentiated-truncation-bound}.
Then
\begin{align}
\|A_{\tau}\eta-\rho\|_{\mathcal E^*,N}
&\lesssim
\bigl((1+T^{1/2})\delta^{1/2}h^{r-1}
+\delta^{-1/2}h^r\bigr)\mathcal N_{r,2}(U)
\label{eq:dual-residual-bound-space}\\
&\quad+(1+T^{1/2})\tau^\nu
\mathcal C^{\sharp}_{T,\nu+2}(U,F)
\label{eq:dual-residual-bound-time}
\end{align}
The hidden constant depends on the fixed time stencil and on the coercivity
constants, but not on \(h\), \(\tau\), or \(\delta\).
\end{prop}
\begin{proof}
Because the projection is time independent, the projection error and all its
time differences and method averages are \(L^2\)-orthogonal to
\(\mathcal V_h\). Expanding
\(P_{\delta}=\mathcal M_{\tau}+\delta A_{\tau}
+\delta^2GD_{\tau}^{-}\) therefore gives the exact identity
\begin{equation}
(A_{\tau}\eta^n,P_{\delta}V^n)_{\Omega}
=(\delta A_{\tau}\eta^n-\mathcal M_{\tau}\eta^n,
A_{\tau}V^n)_{\Omega}
+\delta^2(A_{\tau}\eta^n,GD_{\tau}^{-}V^n)_{\Omega}
\label{eq:dual-projection-functional-identity}
\end{equation}
The first term is controlled directly by the dissipation quantity
\(\delta^{1/2}A_{\tau}V\). In the last term,
\(GD_{\tau}^{-}V=D_{\tau}^{-}GV\), so discrete summation by parts moves the
time difference onto \(A_{\tau}\eta\), with both endpoint terms retained.
The same decomposition applied to \((\rho,P_{\delta}V)_{\Omega}\) moves the
last time difference onto \(\rho\). Coercivity controls
\(\delta^{1/2}A_{\tau}V\), while
\eqref{eq:test-operator-norm} controls \(V\) and \(\delta GV\) at every
computed and start-up level. Approximation and differentiated consistency then
give \eqref{eq:dual-residual-bound-space}--\eqref{eq:dual-residual-bound-time}.
The complete horizon-uniform calculation is given in Appendix
\ref{app:dual-residual-proof}.
\end{proof}

\begin{thm}[Abstract error estimate]
\label{thm:meta-half}
Assume Theorem \ref{thm:meta-stability}, the regularity hypotheses of
Proposition \ref{prop:dual-residual-bound}, and projected start-up data
\(U_h^j=\Pi_hU(t^j)\) for \(0\le j<n_0\). Then
\begin{align}
\max_{n_0\le n\le N}\|T_{\delta}(U_h^n-U(t^n))\|_{\Omega}
&\lesssim (h^r+\delta h^{r-1})
\|U\|_{L^{\infty}(0,T;H^r(\Omega))} \label{eq:abstract-error-estimate-space}\\
&\quad+
\bigl((1+T^{1/2})\delta^{1/2}h^{r-1}+\delta^{-1/2}h^r\bigr)
\mathcal N_{r,2}(U) \label{eq:abstract-error-estimate-dual}\\
&\quad+(1+T^{1/2})\tau^\nu\mathcal C^{\sharp}_{T,\nu+2}(U,F)
\label{eq:abstract-error-estimate-truncation}\\
&\quad+\tau^{1/2}\bigl(h^r+\tau h^{r-1}\bigr)
\|\partial_tU\|_{L^2(0,T;H^r(\Omega))}
\label{eq:abstract-error-estimate-strong-time}\\
&\quad+\tau^{1/2}h^{r-1}
\|U\|_{L^2(0,T;H^r(\Omega))}
\label{eq:abstract-error-estimate-strong-space}\\
&\quad+\tau^{\nu+1/2}\mathcal C^{\sharp}_{T,\nu+2}(U,F)
\label{eq:abstract-error-estimate-time}
\end{align}
The same right-hand side controls the accumulated material-residual
dissipation in the form
\(\bigl(\tau\delta\sum_n\|A_{\tau}\phi^n\|_{\Omega}^2\bigr)^{1/2}\).
\end{thm}
\begin{proof}
Apply Theorem \ref{thm:meta-stability} to \eqref{eq:error-equation}. The
start-up term vanishes, and
\eqref{eq:discrete-error-residual-splitting},
\eqref{eq:strong-residual-bound-space}--\eqref{eq:strong-residual-bound-time},
and Proposition \ref{prop:dual-residual-bound}
bound the strong and dual residual contributions. Finally,
\begin{equation}
U_h^n-U(t^n)=\phi^n-\eta^n \label{eq:total-error-identity}
\end{equation}
and \eqref{eq:spatial-approximation} bounds the graph norm of \(\eta^n\).
More precisely, the time-uniform projection contribution is
\begin{equation}
\max_{0\le n\le N}\|T_{\delta}\eta^n\|_{\Omega}
\lesssim(h^r+\delta h^{r-1})
\|U\|_{L^{\infty}(0,T;H^r(\Omega))} \label{eq:projection-graph-norm-bound}
\end{equation}
The terms in \eqref{eq:abstract-error-estimate-strong-time}--
\eqref{eq:abstract-error-estimate-time} come from the explicit
\(\tau^{1/2}\|R\|_{\ell_{\tau}^2}\) term in
\eqref{eq:abstract-stability-estimate}. In particular,
\eqref{eq:abstract-error-estimate-strong-time} contains the additional term
\(\tau^{3/2}h^{r-1}
\|\partial_tU\|_{L^2(0,T;H^r(\Omega))}\); no parameter relation is used to
discard it.
\end{proof}


\subsection{Parameter Balance}
\label{subsec:parameter-balance}
We now specialise the estimate to the parameter regimes used later.
For the standard stabilised choice \(\delta\sim h\), the spatial part of
\eqref{eq:abstract-error-estimate-space}--\eqref{eq:abstract-error-estimate-time}
is of order \(h^{r-1/2}\) when \(\tau\lesssim h\). The projection term is
then of the higher order \(h^r\). Thus, with \(r=k+1\), smooth degree-\(k\)
approximation gives order \(h^{k+1/2}\) under the stated regularity
assumptions.

The AM3--AM5 assumptions in Section \ref{sec:adams-moulton} and the
AB3--AB4 assumption in Section \ref{sec:ab} already include
\(\tau\lesssim h\). For those methods,
\(\tau h^{r-1}\|\partial_tU\|_{L^2(0,T;H^r(\Omega))}\) is therefore
controlled by
\(h^r\|\partial_tU\|_{L^2(0,T;H^r(\Omega))}\), so the previously stated
asymptotic orders are unchanged. No condition \(\tau\lesssim h\) is added
to the abstract stability theorem or to \(\theta\)-method stability. The
\(\theta\)-method error estimate retains the explicit
\(\tau^{3/2}h^{r-1}\)-term unless a separately stated parameter regime
justifies its absorption. In particular, under
\(\delta\sim h\) and \(\tau\lesssim h\), the
\(h^{r-1/2}\)-conclusion is unchanged.

For the normal-equation choice \(\delta=b_0\tau\), the two dual contributions
are \(\tau^{1/2}h^{r-1}\) and \(\tau^{-1/2}h^r\). Formally balancing the latter
with the time error gives
\begin{equation}
\tau=h^{r/(\nu+1/2)} \label{eq:normal-equation-parameter-balance}
\end{equation}
This is an accuracy balance; stability must still be checked for the chosen
time-stepping method. The Adams--Moulton theorem below includes the
normal-equation choice under its method-dependent CFL conditions, while
Proposition \ref{prop:dual-residual-bound} applies in both parameter regimes.


\section{An A-Stable Scheme: The Theta Method}
\label{sec:theta-method}
The \(\theta\)-family gives the simplest realisation of the abstract framework: its
energy balance follows from a one-step identity and requires no CFL condition.
It includes Crank--Nicolson at \(\theta=1/2\) and backward Euler at
\(\theta=1\). Throughout this section, let \(1/2\le\theta\le1\) and set
\begin{align}
D_{\tau}^{-}U^n&=\tau^{-1}(U^n-U^{n-1}) \label{eq:theta-time-difference}\\
\mathcal M_{\theta}U^n&=\theta U^n+(1-\theta)U^{n-1} \label{eq:theta-method-average}\\
\delta&=\theta\tau \label{eq:theta-stabilisation-parameter}
\end{align}
Here the current method coefficient is \(b_0=\theta\), so the last choice is
the normal-equation value \(b_0\tau\). The corresponding acoustic-wave
realisation is given in Section \ref{subsec:acoustic-wave-equation}.

\paragraph{Energy and Dissipation.}
The energy is one half of the squared graph norm associated with the normal
equation.
Define
\begin{equation}
e_{\tau}^n(U)=\frac12\|T_{\delta}U^n\|_{\Omega}^2
=\frac12\|U^n\|_{\Omega}^2
+\frac12\delta^2\|GU^n\|_{\Omega}^2 \label{eq:theta-energy}
\end{equation}
and
\begin{align}
d_{\tau}^n(U)
&=\delta\|A_{\tau}U^n\|_{\Omega}^2
+\frac{2\theta-1}{2}\tau\|D_{\tau}^{-}U^n\|_{\Omega}^2 \label{eq:theta-dissipation-solution}\\
&\quad+\frac{2\theta-1}{2}\tau\delta^2
\|GD_{\tau}^{-}U^n\|_{\Omega}^2 \label{eq:theta-dissipation-graph}
\end{align}
Both quantities are nonnegative for \(1/2\le\theta\le1\). Consequently,
\eqref{eq:energy-coercivity}--\eqref{eq:dissipation-coercivity} hold with the
common constant \(c_E=1/2\). At the Crank--Nicolson endpoint \(\theta=1/2\),
the two terms proportional to \(2\theta-1\) vanish, while the material-residual term
\(\delta\|A_{\tau}U^n\|_{\Omega}^2\) remains.

\paragraph{Increment Identities.}
The elementary Hilbert-space identity for the \(\theta\)-average gives
\begin{align}
\tau(D_{\tau}^{-}U^n,\mathcal M_{\theta}U^n)_{\Omega}
&=\frac12\bigl(\|U^n\|_{\Omega}^2-\|U^{n-1}\|_{\Omega}^2\bigr) \label{eq:theta-increment-energy}\\
&\quad+\frac{2\theta-1}{2}\tau^2
\|D_{\tau}^{-}U^n\|_{\Omega}^2 \label{eq:theta-increment-dissipation}
\end{align}
Applying the same identity to \(GU^n\) yields
\begin{align}
\tau\delta^2(G\mathcal M_{\theta}U^n,
GD_{\tau}^{-}U^n)_{\Omega}
&=\frac12\delta^2
\bigl(\|GU^n\|_{\Omega}^2-\|GU^{n-1}\|_{\Omega}^2\bigr) \label{eq:theta-graph-increment-energy}\\
&\quad+\frac{2\theta-1}{2}\tau^2\delta^2
\|GD_{\tau}^{-}U^n\|_{\Omega}^2 \label{eq:theta-graph-increment-dissipation}
\end{align}
The two increment identities convert the first and third terms in the
fundamental algebraic identity into telescoping energy differences and
nonnegative dissipation. In particular, the factor \(\tau\delta^2\) in
\eqref{eq:theta-graph-increment-energy} matches
\eqref{eq:abstract-tested-expansion-graph}.

\begin{prop}[Energy--dissipation identity for the Theta method]
\label{prop:theta-energy}
The solution of \eqref{eq:supg-weak} with
\eqref{eq:theta-time-difference}--\eqref{eq:theta-stabilisation-parameter}
satisfies
\begin{equation}
e_{\tau}^n(U)-e_{\tau}^{n-1}(U)+\tau d_{\tau}^n(U)
=(\tau\mathcal M_{\theta}F^n,P_{\delta}U^n)_{\Omega} \label{eq:theta-energy-balance}
\end{equation}
\end{prop}
\begin{proof}
Use \(S_{\delta}U^n\) as the test function in \eqref{eq:supg-weak}.
Equation \eqref{eq:stability_meta}--\eqref{eq:abstract-tested-expansion-graph}
gives
\begin{align}
(\tau A_{\tau}U^n,P_{\delta}U^n)_{\Omega}
&=\tau(D_{\tau}^{-}U^n,\mathcal M_{\theta}U^n)_{\Omega}
+\tau\delta\|A_{\tau}U^n\|_{\Omega}^2 \label{eq:theta-balance-expansion-main}\\
&\quad+\tau\delta^2(G\mathcal M_{\theta}U^n,
GD_{\tau}^{-}U^n)_{\Omega} \label{eq:theta-balance-expansion-graph}
\end{align}
Substitution of
\eqref{eq:theta-increment-energy}--\eqref{eq:theta-graph-increment-dissipation}
and the definitions \eqref{eq:theta-energy}--\eqref{eq:theta-dissipation-graph}
proves \eqref{eq:theta-energy-balance}.
\end{proof}

\begin{cor}[Theta-method stability]
\label{cor:theta-stability}
For \(1/2\le\theta\le1\), the energy hypothesis
\eqref{eq:energy-coercivity}--\eqref{eq:abstract-energy-hypothesis-defect}
holds with the forcing-defect term omitted.
Consequently,
\begin{align}
\|U_h\|_{\mathcal E,N}^2
&\lesssim\|T_{\delta}U_h^0\|_{\Omega}^2
+\|\mathcal M_{\theta}F\|_{\mathcal E^*,N}^2 \label{eq:theta-stability-estimate}
\end{align}
\end{cor}
\begin{proof}
Sum \eqref{eq:theta-energy-balance} from \(n=1\) to an arbitrary horizon
\(m\). This is the abstract energy hypothesis with \(n_0=1\),
\(R^n=\mathcal M_{\theta}F^n\), and no strong forcing defect. Apply Theorem
\ref{thm:meta-stability} with
\(e_{\tau}^n\) and \(d_{\tau}^n\) defined in
\eqref{eq:theta-energy}--\eqref{eq:theta-dissipation-graph}.
\end{proof}


\subsection{Consistency and Error Estimate}
\label{subsec:theta-consistency-error}
To apply the abstract error theorem, it remains to determine the temporal order
and verify the differentiated consistency bound. For the sampled exact
solution, the differential equation gives
\begin{equation}
\rho^n=D_{\tau}^{-}U(t^n)
-\mathcal M_{\theta}\partial_tU(t^n) \label{eq:theta-truncation-residual}
\end{equation}
Taylor's formula with an integral remainder yields
\begin{align}
\|\rho^n\|_{\Omega}
&\lesssim |2\theta-1|\tau^{1/2}
\|\partial_t^2U\|_{L^2(t^{n-1},t^n;\mathcal H)} \label{eq:theta-first-order-consistency}\\
&\quad+\tau^{3/2}
\|\partial_t^3U\|_{L^2(t^{n-1},t^n;\mathcal H)} \label{eq:theta-second-order-consistency}
\end{align}
Thus the method has order \(\nu=1\) for fixed \(\theta>1/2\), and order
\(\nu=2\) for \(\theta=1/2\) or, more generally,
\(\theta=1/2+O(\tau)\). The dual-residual argument also requires a uniform
bound for \(\rho\) and an \(\ell_{\tau}^2\)-bound for its forward difference.
Applying the same integral-remainder formula to these quantities gives
\begin{align}
\max_{1\le n\le N}\|\rho^n\|_{\Omega}
+\|D_{\tau}^{+}\rho\|_{\ell_{\tau}^2(1,N-1;\mathcal H)}
&\lesssim |2\theta-1|\tau\Bigl(
\|\partial_t^2U\|_{L^{\infty}(0,T;\mathcal H)}
+\|\partial_t^3U\|_{L^2(0,T;\mathcal H)}\Bigr)
\label{eq:theta-differentiated-consistency-first}\\
&\quad+\tau^2\Bigl(
\|\partial_t^3U\|_{L^{\infty}(0,T;\mathcal H)}
+\|\partial_t^4U\|_{L^2(0,T;\mathcal H)}\Bigr)
\label{eq:theta-differentiated-consistency-second}
\end{align}
Hence \eqref{eq:maximum-truncation-bound}--
\eqref{eq:differentiated-truncation-bound} hold with \(\nu=1\) for fixed
\(\theta>1/2\), and with \(\nu=2\) for \(\theta=1/2\) or
\(\theta=1/2+O(\tau)\). The displayed norms are included in
\(\mathcal C^{\sharp}_{T,\nu+2}(U,F)\).

\begin{cor}[Theta-method error estimate]
\label{cor:theta-error}
Assume the regularity hypotheses of Proposition
\ref{prop:dual-residual-bound} and projected initial data. Then
\eqref{eq:abstract-error-estimate-space}--\eqref{eq:abstract-error-estimate-time}
holds with \(\nu=1\) for fixed \(\theta>1/2\), and with \(\nu=2\) for
\(\theta=1/2\) or \(\theta=1/2+O(\tau)\).
\end{cor}
\begin{proof}
Combine Corollary \ref{cor:theta-stability},
\eqref{eq:theta-first-order-consistency}--\eqref{eq:theta-second-order-consistency},
\eqref{eq:theta-differentiated-consistency-first}--
\eqref{eq:theta-differentiated-consistency-second},
and Theorem \ref{thm:meta-half}.
\end{proof}

Under the balanced scaling \(\tau\sim h\), the normal-equation choice satisfies
\(\delta\sim h\). With \(r=k+1\), the parameter balance in Section
\ref{subsec:parameter-balance} therefore yields the graph-norm order
\(O(h^{k+1/2}+\tau^\nu)\), with \(\nu\) as stated in Corollary
\ref{cor:theta-error}.


\section{Adams--Moulton Methods of Orders Three to Five}
\label{sec:adams-moulton}
This section treats the implicit Adams--Moulton methods AM3--AM5. We first
establish consistency and the increment bounds needed to control the
multistep history. Exact energy decompositions then reveal the different
stability mechanisms: AM3 and AM4 require strengthened sufficient CFL
conditions for general polynomial degree, whereas AM5 is stable under the
standard hyperbolic condition. For continuous piecewise affine finite
elements, an elementwise cancellation recovers the standard hyperbolic regime
also for AM3 and AM4. The Adams--Bashforth methods are defined and analysed
independently in Section \ref{sec:ab}.

\paragraph{Method Coefficients.}
For all three methods,
\(D_{\tau}^{-}U^n=\tau^{-1}(U^n-U^{n-1})\), while the coefficients in
\eqref{eq:method-average} are
\begin{align}
\mathrm{AM3}:\quad(b_0,b_1,b_2)
&=\left(\frac5{12},\frac23,-\frac1{12}\right) \label{eq:am3-coefficients}\\
\mathrm{AM4}:\quad(b_0,b_1,b_2,b_3)
&=\left(\frac38,\frac{19}{24},-\frac5{24},\frac1{24}\right) \label{eq:am4-coefficients}\\
\mathrm{AM5}:\quad(b_0,b_1,b_2,b_3,b_4)
&=\left(\frac{251}{720},\frac{646}{720},-\frac{264}{720},
\frac{106}{720},-\frac{19}{720}\right) \label{eq:am5-coefficients}
\end{align}
Set \(n_0(3)=2\), \(n_0(4)=3\), and \(n_0(5)=4\). These are the first
computed indices for AM3, AM4, and AM5, respectively.
Define the backward increments recursively by
\begin{align}
d^1U^n&=U^n-U^{n-1} \label{eq:first-backward-increment}\\
d^jU^n&=d^{j-1}U^n-d^{j-1}U^{n-1}
\quad 2\le j\le4 \label{eq:higher-backward-increments}
\end{align}
With \(\widetilde U^n=(U^n+U^{n-1})/2\), exact expansion of the coefficients
in \eqref{eq:am3-coefficients}--\eqref{eq:am5-coefficients} gives
\begin{align}
\mathcal M_{\mathrm{AM3}}U^n
&=\widetilde U^n-\frac1{12}d^2U^n \label{eq:am3-increment-average}\\
\mathcal M_{\mathrm{AM4}}U^n
&=\widetilde U^n-\frac1{12}d^2U^n-\frac1{24}d^3U^n \label{eq:am4-increment-average}\\
\mathcal M_{\mathrm{AM5}}U^n
&=\widetilde U^n-\frac1{12}d^2U^n-\frac1{24}d^3U^n
-\frac{19}{720}d^4U^n \label{eq:am5-increment-average}
\end{align}
These centred forms isolate the higher backward increments that generate the
bulk terms in the energy decompositions below.


\subsection{Consistency of the Adams--Moulton Methods}
\label{sec:trunc-adams}
Set \(\sigma(t)=F(t)-GU(t)=\partial_tU(t)\). The Adams--Moulton formula is the
interpolatory quadrature rule obtained by approximating this exact time
derivative over one time step. Indeed, the exact equation implies
\begin{equation}
U(t^n)-U(t^{n-1})=\int_{t^{n-1}}^{t^n}\sigma(t)\,dt \label{eq:adams-exact-integral}
\end{equation}
For AM\(\nu\), let \(\sigma_{\nu-1}\) be the polynomial of degree at most
\(\nu-1\) interpolating \(\sigma\) at the nodes
\begin{equation}
t^{n-j}\quad 0\le j\le \nu-1 \label{eq:adams-interpolation-nodes}
\end{equation}
Integrating the interpolant over the current step gives, by construction,
\begin{equation}
\int_{t^{n-1}}^{t^n}\sigma_{\nu-1}(t)\,dt
=\tau\mathcal M_{\mathrm{AM}\nu}\sigma(t^n) \label{eq:adams-interpolatory-quadrature}
\end{equation}
Subtracting \eqref{eq:adams-interpolatory-quadrature} from
\eqref{eq:adams-exact-integral} gives
\begin{align}
\tau\rho^n
&=\int_{t^{n-1}}^{t^n}\bigl(\sigma(t)-\sigma_{\nu-1}(t)\bigr)\,dt \label{eq:adams-truncation-integral}\\
\rho^n
&=D_{\tau}^{-}U(t^n)+G\mathcal M_{\mathrm{AM}\nu}U(t^n)
-\mathcal M_{\mathrm{AM}\nu}F(t^n) \label{eq:adams-truncation-residual}
\end{align}
The interpolation remainder and the Cauchy--Schwarz inequality yield
\begin{equation}
\|\rho^n\|_{\Omega}
\le C_\nu\tau^{\nu-1/2}
\|\partial_t^\nu(F-GU)\|_{L^2(t^{n-\nu+1},t^n;\mathcal H)} \label{eq:adams-truncation-bound}
\end{equation}
Thus \eqref{eq:high-consist} holds with the claimed order
\(\nu\in\{3,4,5\}\). More precisely, in
\eqref{eq:global-truncation-regularity} one may take
\begin{equation}
\mathcal C_{n,\nu+1}(U,F)
=C_\nu\|\partial_t^\nu(F-GU)\|_{L^2(t^{n-\nu+1},t^n;\mathcal H)}
\label{eq:am-local-truncation-regularity}
\end{equation}
These intervals have overlap at most \(\nu\), so
\eqref{eq:global-truncation-bound} follows with a constant depending only on
the fixed order. The dual-residual estimate also requires uniform control of
the residual and square-summable control of its first forward difference.
Writing the interpolation remainder in integral form gives both bounds:
\begin{align}
\max_{n_0(\nu)\le n\le N}\|\rho^n\|_{\Omega}
+\|D_{\tau}^{+}\rho\|_{
\ell_{\tau}^2(n_0(\nu),N-1;\mathcal H)}
&\le C_\nu\tau^\nu\Bigl(
\|\partial_t^\nu(F-GU)\|_{L^{\infty}(0,T;\mathcal H)}
\label{eq:am-differentiated-consistency-main}\\
&\quad+\|\partial_t^{\nu+1}(F-GU)\|_{L^2(0,T;\mathcal H)}\Bigr)
\label{eq:am-differentiated-consistency-derivative}
\end{align}
Thus \eqref{eq:maximum-truncation-bound}--
\eqref{eq:differentiated-truncation-bound} hold for AM3--AM5 with these
two norms included in \(\mathcal C^{\sharp}_{T,\nu+2}(U,F)\).


\subsection{Stability Assumptions and Increment Bounds}
\label{subsec:adams-stability-assumptions}
Assume the inverse inequality
\begin{equation}
\|GV_h\|_{\Omega}\le C_Gh^{-1}\|V_h\|_{\Omega}
\quad \forall V_h\in\mathcal V_h \label{eq:adams-inverse-inequality}
\end{equation}
and the parameter relations
\begin{equation}
0<\delta\le C_{\delta}h,\qquad
\tau\le c_0C_G^{-1}h \label{eq:cflcond}
\end{equation}
where \(c_0>0\) will be chosen sufficiently small. This is the standard
hyperbolic CFL restriction used for all three methods. For the general-degree
AM3 and AM4 estimates, the present proof also uses the respective additional
conditions
\begin{align}
\tau&\le C_{4/3}h^{4/3} &&\text{for AM3} \label{eq:am3-strengthened-cfl}\\
\tau&\le C_{6/5}h^{6/5} &&\text{for AM4} \label{eq:am4-strengthened-cfl}
\end{align}
No smallness is imposed on \(C_{4/3}\) or \(C_{6/5}\); these constants enter
the final Gronwall constant. AM5 requires no strengthened condition.

\paragraph{Relation to the Normal Equation.}
Only the upper bound on \(\delta\) is used below. Together with the inverse
inequality, it makes the graph norm uniformly equivalent to the
\(L^2\)-norm on \(\mathcal V_h\); no lower bound proportional to \(h\) is
needed. The stability proof therefore permits \(\delta/h\to0\). In
particular, the normal-equation value \(\delta=b_0\tau\) satisfies
\eqref{eq:cflcond} under the relevant time-step condition and is covered for
AM3--AM5. Choices \(\delta\ne b_0\tau\) remain generally nonsymmetric SUPG
schemes.

The energy decompositions contain negative bulk terms involving the highest
backward increments for AM3 and AM4, and a negative fourth-increment term for
AM5. The following lemma relates these increments to earlier solution values
and to the forcing.

\begin{lem}[Increment recurrence]
\label{lem:adams-increment-recurrence}
Let \(\nu\in\{3,4,5\}\), and let \(V_h\) solve the generally forced AM\(\nu\)
equation
\begin{equation}
(A_{\tau}V_h^n,T_{\delta}W)_{\Omega}
=(R^n,T_{\delta}W)_{\Omega}
\quad\forall W\in\mathcal V_h,\qquad n\ge n_0(\nu)
\label{eq:am-general-forced-equation}
\end{equation}
Define
the remaining-stencil coefficients by
\begin{align}
\gamma_0&=b_0+b_1 \label{eq:adams-recurrence-first-coefficient}\\
\gamma_q&=b_{q+1}\quad 1\le q\le \nu-2 \label{eq:adams-recurrence-remaining-coefficients}
\end{align}
For every admissible increment order \(j\ge1\) and every
\(n\ge n_0(\nu)+j-1\),
\begin{align}
\|d^jV_h^n\|_{\delta}
&\lesssim\frac{\tau}{h}
\sum_{q=0}^{\nu-2}\|d^{j-1}V_h^{n-1-q}\|_{\delta} \label{eq:adams-increment-recurrence-solution}\\
&\quad+\tau\sum_{q=0}^{j-1}
\|R^{n-q}\|_{\Omega} \label{eq:adams-increment-recurrence-forcing}
\end{align}
where \(d^0V_h^n=V_h^n\). In particular,
\begin{equation}
\|d^1V_h^n\|_{\delta}
\lesssim\frac{\tau}{h}\sum_{q=0}^{\nu-2}\|V_h^{n-1-q}\|_{\delta}
+\tau\|R^n\|_{\Omega} \label{eq:d1ubound}
\end{equation}
The local recurrence above is the basic estimate. Its iteration gives the
bulk and endpoint bounds stated next.
Early increments for which the formal difference would precede the first
computed equation are estimated directly from the computed updates. Only
their solution terms are absorbed into \(\mathcal I_{\tau}(V_h)\); every
forcing contribution is retained from level \(n_0(\nu)\) onward. Consequently,
for \(\nu=3,4\),
\begin{align}
\sum_{n=\nu-1}^{m}\|d^{\nu-1}V_h^n\|_{\delta}^2
&\lesssim\tau C_{2(\nu-1)/(2\nu-3)}^{\,2\nu-3}
\sum_{n=0}^{m-1}\|V_h^n\|_{\delta}^2 \label{eq:adams-high-increment-bulk-solution}\\
&\quad+\tau^2\sum_{n=n_0(\nu)}^{m}\|R^n\|_{\Omega}^2
+C\mathcal I_{\tau}(V_h) \label{eq:adams-high-increment-bulk-forcing}
\end{align}
and, for AM5,
\begin{align}
\sum_{n=4}^{m}\|d^4V_h^n\|_{\delta}^2
&\lesssim c_0^2\sum_{n=4}^{m-1}\|d^3V_h^n\|_{\delta}^2 \label{eq:am5-fourth-increment-control-solution}\\
&\quad+\tau^2\sum_{n=n_0(5)}^{m}\|R^n\|_{\Omega}^2
+C\mathcal I_{\tau}(V_h) \label{eq:am5-fourth-increment-control-forcing}
\end{align}
Moreover, every horizon \(m\ge n_0(\nu)\) satisfies the endpoint bounds
\begin{align}
\sum_{j=1}^{\nu-2}\|d^jV_h^m\|_{\delta}^2
&\lesssim c_0^2\max_{0\le n\le m}\|V_h^n\|_{\delta}^2
+\tau^2\sum_{n=n_0(\nu)}^m\|R^n\|_{\Omega}^2 \label{eq:adams-endpoint-increment-solution}\\
&\quad+C\mathcal I_{\tau}(V_h) \label{eq:adams-endpoint-increment-startup}
\end{align}
and, for AM5,
\begin{align}
\|d^2V_h^{m-1}\|_{\delta}^2
&\lesssim c_0^2\max_{0\le n\le m}\|V_h^n\|_{\delta}^2
+\tau^2\sum_{n=4}^m\|R^n\|_{\Omega}^2
+C\mathcal I_{\tau}(V_h) \label{eq:am5-previous-endpoint-increment}
\end{align}
\end{lem}
\begin{proof}
We first derive the local recurrence, then sum its iterates to obtain the
bulk estimates, and finally apply it at a fixed horizon to control the
boundary increments.
Apply \(d^{j-1}\) to \eqref{eq:am-general-forced-equation} and multiply by
\(\tau\). Since
\(d^{j-1}V_h^n=d^{j-1}V_h^{n-1}+d^jV_h^n\), the current \(b_0\) term can be
collected exactly in the left-hand operator. The resulting equation is
\begin{align}
(T_{b_0\tau}d^jV_h^n,T_{\delta}W)_{\Omega}
&=-\tau\left(G\sum_{q=0}^{\nu-2}\gamma_q
d^{j-1}V_h^{n-1-q},T_{\delta}W\right)_{\Omega} \label{eq:adams-differenced-equation-history}\\
&\quad+\tau(d^{j-1}R^n,T_{\delta}W)_{\Omega}
\quad \forall W\in\mathcal V_h \label{eq:adams-differenced-equation-forcing}
\end{align}
This displays the previous-value stencil explicitly. With \(W=d^jV_h^n\),
skew-symmetry makes the current increment coercive:
\begin{equation}
(T_{b_0\tau}d^jV_h^n,T_{\delta}d^jV_h^n)_{\Omega}
=\|d^jV_h^n\|_{\Omega}^2
+b_0\tau\delta\|Gd^jV_h^n\|_{\Omega}^2 \label{eq:adams-current-increment-coercivity}
\end{equation}
The inverse inequality and the upper bound \(\delta\le C_{\delta}h\) give
the norm equivalence
\begin{equation}
\|V_h\|_{\Omega}\le\|V_h\|_{\delta}
\lesssim\|V_h\|_{\Omega}\quad \forall V_h\in\mathcal V_h \label{eq:adams-graph-l2-equivalence}
\end{equation}
Equations \eqref{eq:adams-differenced-equation-history}--\eqref{eq:adams-graph-l2-equivalence}
therefore prove the local estimates
\eqref{eq:adams-increment-recurrence-solution}--\eqref{eq:d1ubound}.

Iterating the previous-value recurrence \(\nu-1\) times gives the solution
factor \((\tau/h)^{\nu-1}\). Every resulting solution value has a time index
strictly smaller than that of the highest increment. After squaring and
summing through \(m\), this gives the sum only through \(m-1\) in
\eqref{eq:adams-high-increment-bulk-solution}. Finite convolution of the
forcing values retains precisely the computed levels \(n_0(\nu),\ldots,m\).
The exact power relations
\begin{align}
(\tau/h)^4&\le C_{4/3}^{\,3}\tau
&&\text{under \eqref{eq:am3-strengthened-cfl}} \label{eq:am3-cfl-power}\\
(\tau/h)^6&\le C_{6/5}^{\,5}\tau
&&\text{under \eqref{eq:am4-strengthened-cfl}} \label{eq:am4-cfl-power}
\end{align}
prove \eqref{eq:adams-high-increment-bulk-solution}--\eqref{eq:adams-high-increment-bulk-forcing}.
For AM5, applying the recurrence once with \(j=4\) gives only the previous
third differences \(d^3V_h^{n-1-q}\). Squaring, summing, and using
\(\tau/h\le c_0C_G^{-1}\) proves
\eqref{eq:am5-fourth-increment-control-solution}--\eqref{eq:am5-fourth-increment-control-forcing}.
Finally, a finite iteration of the same previous-value recurrence at a fixed
horizon proves \eqref{eq:adams-endpoint-increment-solution}--\eqref{eq:adams-endpoint-increment-startup}.
For AM5, applying the order-two endpoint argument at horizon \(m-1\) and
using the update at level \(m\) proves \eqref{eq:am5-previous-endpoint-increment}.
\end{proof}


For continuous piecewise affine functions and constant-coefficient \(G\),
\(GV_h\) is elementwise constant. The elementwise mean of \(d^1V_h^n\)
therefore drops out after applying \(G\), avoiding the recurrence that
produced the strengthened CFL powers above.

\begin{lem}[Piecewise affine increment bound]
\label{lem:adams-piecewise-affine-increment-bound}
Assume that \(\mathcal V_h\) consists of continuous piecewise affine
functions on a shape-regular quasi-uniform mesh and that \(G\) has constant
coefficients, so that
\begin{equation}
GV_h\vert_K\in[\mathbb P_0(K)]^{n_{\mathrm{sys}}}
\quad\forall V_h\in\mathcal V_h,\quad K\in\mathcal T_h
\label{eq:adams-piecewise-affine-mapping}
\end{equation}
Let \(V_h\) solve \eqref{eq:am-general-forced-equation} for
\(\nu\in\{3,4\}\), and assume
\begin{equation}
\tau\le C_{\tau\delta}\delta,\qquad
0<\delta\le C_{\delta}h,\qquad
\tau\le c_0C_G^{-1}h
\label{eq:adams-piecewise-affine-parameters}
\end{equation}
Then every admissible horizon satisfies
\begin{align}
\sum_{n=\nu-1}^{m}\|d^{\nu-1}V_h^n\|_{\delta}^2
&\lesssim C_{\tau\delta}\left(\frac{\tau}{h}\right)^2
\tau\delta\sum_{n=n_0(\nu)}^{m}\|A_{\tau}V_h^n\|_{\Omega}^2
\label{eq:adams-piecewise-affine-increment-solution}\\
&\quad+\tau^2\sum_{n=n_0(\nu)}^{m}\|R^n\|_{\Omega}^2
+C\mathcal I_{\tau}(V_h)
\label{eq:adams-piecewise-affine-increment-forcing}
\end{align}
The hidden constant is independent of \(h\), \(\tau\), and the horizon.
\end{lem}
\begin{proof}
Let \(Q_0\) be the elementwise \(L^2\)-projection onto discontinuous
piecewise constants, let \(\Pi_h\) be the \(L^2\)-projection onto
\(\mathcal V_h\), and set
\begin{equation}
Z^n=(\mathrm{Id}-Q_0)d^1V_h^n
\label{eq:adams-piecewise-affine-fluctuation}
\end{equation}
Thus \(Z^n\) is precisely the elementwise fluctuation of the first increment.
Multiplication of \eqref{eq:am-general-forced-equation} by \(\tau\) gives
\begin{equation}
(d^1V_h^n+\tau G\mathcal M_{\mathrm{AM}\nu}V_h^n,
T_{\delta}W)_{\Omega}
=\tau(R^n,T_{\delta}W)_{\Omega}
\label{eq:adams-piecewise-affine-multiplied-equation}
\end{equation}
Choose \(W=\Pi_hZ^n\). Since \(d^1V_h^n\in\mathcal V_h\), projection
orthogonality gives
\((d^1V_h^n,\Pi_hZ^n)_{\Omega}=\|Z^n\|_{\Omega}^2\).
Moreover, \(G\mathcal M_{\mathrm{AM}\nu}V_h^n\) is elementwise constant and
is orthogonal to \(Z^n\), while
\(D_{\tau}^{-}V_h^n\) is orthogonal to \(\Pi_hZ^n-Z^n\). Combining the
two terms containing \(G\Pi_hZ^n\) into \(A_{\tau}V_h^n\) therefore gives
the exact identity
\begin{align}
\|Z^n\|_{\Omega}^2
&=\tau(R^n,T_{\delta}\Pi_hZ^n)_{\Omega}
-\tau(A_{\tau}V_h^n,\Pi_hZ^n-Z^n)_{\Omega}
\label{eq:adams-piecewise-affine-fluctuation-identity-main}\\
&\quad-\tau\delta(A_{\tau}V_h^n,G\Pi_hZ^n)_{\Omega}
\label{eq:adams-piecewise-affine-fluctuation-identity-streamline}
\end{align}
The \(L^2\)-stability of \(\Pi_h\), the inverse inequality, and
\(\delta\lesssim h\) show that
\begin{equation}
\|Z^n\|_{\Omega}
\lesssim\tau\bigl(\|A_{\tau}V_h^n\|_{\Omega}
+\|R^n\|_{\Omega}\bigr)
\label{eq:adams-piecewise-affine-fluctuation-bound}
\end{equation}

Take one backward difference of
\eqref{eq:adams-piecewise-affine-multiplied-equation}. At every level for
which both equations are computed,
\begin{equation}
(d^2V_h^n+\tau G\mathcal M_{\mathrm{AM}\nu}d^1V_h^n,
T_{\delta}W)_{\Omega}
=\tau(d^1R^n,T_{\delta}W)_{\Omega}
\label{eq:adams-piecewise-affine-differenced-equation}
\end{equation}
With \(W=d^2V_h^n\), skew-symmetry removes
\((d^2V_h^n,Gd^2V_h^n)_{\Omega}\). On each element,
\(Gd^1V_h^n=GZ^n\), and the affine inverse estimate gives
\begin{equation}
\|G\mathcal M_{\mathrm{AM}\nu}d^1V_h^n\|_{\Omega}
\lesssim h^{-1}\sum_{j=0}^{\nu-1}\|Z^{n-j}\|_{\Omega}
\label{eq:adams-piecewise-affine-local-inverse}
\end{equation}
Using \eqref{eq:adams-piecewise-affine-fluctuation-bound}, squaring, and
summing in time yields the following estimate. Here the finite difference
\(d^1R^n\) is bounded by its two adjacent values, and the fixed Adams
stencil has uniformly bounded overlap:
\begin{align}
\sum_{n=n_0(\nu)+1}^m\|d^2V_h^n\|_{\delta}^2
&\lesssim\frac{\tau^4}{h^2}
\sum_{n=n_0(\nu)}^m\|A_{\tau}V_h^n\|_{\Omega}^2
+\tau^2\sum_{n=n_0(\nu)}^m\|R^n\|_{\Omega}^2
\label{eq:adams-piecewise-affine-second-increment-bound}\\
&\quad+C\mathcal I_{\tau}(V_h)
\notag
\end{align}
The finitely many differences touching the start-up stencil are included in
\(\mathcal I_{\tau}(V_h)\); their forcing terms are retained in the displayed
forcing sum. For AM3, this already controls the required second increment;
for AM4, the third increment is a finite difference of \(d^2V_h\). Finally,
\(\tau\le C_{\tau\delta}\delta\) converts
\(\tau^4h^{-2}\) into at most
\(C_{\tau\delta}(\tau/h)^2\tau\delta\). This proves
\eqref{eq:adams-piecewise-affine-increment-solution}--
\eqref{eq:adams-piecewise-affine-increment-forcing}.
\end{proof}


\subsection{Exact AM3--AM5 Energy Decompositions}
\label{subsec:adams-energy-decompositions}
The signs of the bulk terms are decisive. AM3 and AM4 produce negative terms
in their highest relevant increment. AM5 instead produces a positive
third-increment term that can absorb its negative fourth-increment term under
the standard CFL restriction.
Recall the first computed indices \(n_0(3)=2\), \(n_0(4)=3\), and
\(n_0(5)=4\), and define
\begin{equation}
\mathcal E_N^{(\nu)}
=\sum_{n=n_0(\nu)}^N
(d^1U^n,\mathcal M_{\mathrm{AM}\nu}U^n)_{\delta} \label{eq:adams-discrete-energy}
\end{equation}
The common identity \eqref{eq:stability_meta}--\eqref{eq:abstract-tested-expansion-graph}
then gives
\begin{equation}
\tau\sum_{n=n_0(\nu)}^N
(A_{\tau}U^n,P_{\delta}U^n)_{\Omega}
=\mathcal E_N^{(\nu)}
+\tau\delta\sum_{n=n_0(\nu)}^N\|A_{\tau}U^n\|_{\Omega}^2 \label{eq:adams-tested-energy-identity}
\end{equation}

\begin{lem}[Exact Adams--Moulton decompositions]
\label{lem:adams-exact-energy-decompositions}
For \(\nu=3,4,5\),
\begin{equation}
\mathcal E_N^{(\nu)}=\mathcal B_N^{(\nu)}+\mathrm{BT}_N^{(\nu)} \label{eq:adams-bulk-boundary-decomposition}
\end{equation}
where the bulk and boundary terms are as follows.

\paragraph{AM3.}
\begin{align}
\mathcal B_N^{(3)}
&=-\frac1{24}\sum_{n=2}^N\|d^2U^n\|_{\delta}^2 \label{eq:am3-energy-bulk}\\
\mathrm{BT}_N^{(3)}
&=\frac12\bigl(\|U^N\|_{\delta}^2-\|U^1\|_{\delta}^2\bigr)
-\frac1{24}\bigl(\|d^1U^N\|_{\delta}^2
-\|d^1U^1\|_{\delta}^2\bigr) \label{eq:am3-energy-boundary}
\end{align}

\paragraph{AM4.}
\begin{align}
\mathcal B_N^{(4)}
&=-\frac1{48}\sum_{n=3}^N\|d^3U^n\|_{\delta}^2 \label{eq:am4-energy-bulk}\\
\mathrm{BT}_N^{(4)}
&=\frac12\bigl(\|U^N\|_{\delta}^2-\|U^2\|_{\delta}^2\bigr)
-\frac1{24}\bigl(\|d^1U^N\|_{\delta}^2
-\|d^1U^2\|_{\delta}^2\bigr) \label{eq:am4-energy-boundary-main}\\
&\quad-\frac1{24}\bigl((d^1U^N,d^2U^N)_{\delta}
-(d^1U^2,d^2U^2)_{\delta}\bigr)
+\frac1{48}\bigl(\|d^2U^2\|_{\delta}^2
-\|d^2U^N\|_{\delta}^2\bigr) \label{eq:am4-energy-boundary-increments}
\end{align}

\paragraph{AM5.}
\begin{align}
\mathcal B_N^{(5)}
&=\frac3{160}\sum_{n=4}^{N-1}\|d^3U^n\|_{\delta}^2
-\frac{19}{1440}\sum_{n=4}^{N}\|d^4U^n\|_{\delta}^2 \label{eq:am5-energy-bulk}\\
\mathrm{BT}_N^{(5)}
&=\frac12\bigl(\|U^N\|_{\delta}^2-\|U^3\|_{\delta}^2\bigr)
-\frac1{24}\bigl(\|d^1U^N\|_{\delta}^2
-\|d^1U^3\|_{\delta}^2\bigr) \label{eq:am5-energy-boundary-main}\\
&\quad-\frac1{24}(d^1U^N,d^2U^N)_{\delta}
+\frac1{24}(d^1U^3,d^2U^3)_{\delta} \notag\\
&\quad-\frac1{48}\|d^2U^N\|_{\delta}^2
+\frac1{48}\|d^2U^3\|_{\delta}^2
+\frac{19}{1440}\|d^2U^{N-1}\|_{\delta}^2
-\frac{19}{1440}\|d^2U^2\|_{\delta}^2 \notag\\
&\quad-\frac{19}{720}(d^1U^N,d^3U^N)_{\delta}
+\frac{19}{720}(d^1U^3,d^3U^3)_{\delta} \notag\\
&\quad-\frac{11}{1440}\|d^3U^N\|_{\delta}^2
+\frac{19}{720}\|d^3U^3\|_{\delta}^2 \label{eq:am5-energy-boundary-increments}
\end{align}
\end{lem}
\begin{proof}
Substitute
\eqref{eq:am3-increment-average}--\eqref{eq:am5-increment-average} into
\eqref{eq:adams-discrete-energy}. The required summation identities are
\begin{align}
\sum_{n=2}^N(d^1U^n,d^2U^n)_{\delta}
&=\frac12\bigl(\|d^1U^N\|_{\delta}^2
-\|d^1U^1\|_{\delta}^2\bigr)
+\frac12\sum_{n=2}^N\|d^2U^n\|_{\delta}^2 \label{eq:du-d2u}\\
\sum_{n=3}^N(d^1U^n,d^3U^n)_{\delta}
&=(d^1U^N,d^2U^N)_{\delta}
-(d^1U^2,d^2U^2)_{\delta} \label{eq:adams-sbp-third-endpoints}\\
&\quad-\sum_{n=3}^{N-1}\|d^2U^n\|_{\delta}^2
-\frac12\bigl(\|d^2U^2\|_{\delta}^2
+\|d^2U^N\|_{\delta}^2\bigr)
+\frac12\sum_{n=3}^N\|d^3U^n\|_{\delta}^2 \label{eq:adams-sbp-third-bulk}\\
\sum_{n=4}^N(d^1U^n,d^4U^n)_{\delta}
&=(d^1U^N,d^3U^N)_{\delta}
-(d^1U^3,d^3U^3)_{\delta} \label{eq:sbp4}\\
&\quad-\frac12\bigl(\|d^2U^{N-1}\|_{\delta}^2
-\|d^2U^2\|_{\delta}^2\bigr)
-\|d^3U^3\|_{\delta}^2
-\frac12\|d^3U^N\|_{\delta}^2 \notag\\
&\quad-\frac32\sum_{n=4}^{N-1}\|d^3U^n\|_{\delta}^2
+\frac12\sum_{n=4}^N\|d^4U^n\|_{\delta}^2 \label{eq:adams-sbp-fourth-bulk}
\end{align}
Collecting like terms gives
\eqref{eq:am3-energy-bulk}--\eqref{eq:am5-energy-boundary-increments}.
In particular, the AM4 bulk coefficient is \(-1/48\), and its initial
energy is at level \(2\). These identities also display why AM5 has a
different stability mechanism from AM3 and AM4.
\end{proof}


\subsection{Lower Bounds for the Energy Decompositions}
\label{subsec:adams-energy-lower-bounds}
The stability proof requires a lower bound for \(\mathcal E_N^{(\nu)}\).
We therefore estimate its bulk and boundary contributions separately and
then recombine them in the stability theorem. The direction of the bulk
estimate is essential.

\begin{lem}[Bulk lower bounds]
\label{lem:adams-bulk-lower-bounds}
Let \(V_h\) solve \eqref{eq:am-general-forced-equation}, and evaluate the
bulk terms on this sequence. Under \eqref{eq:cflcond} and the relevant
strengthened condition, the AM3 and AM4 bulk terms satisfy
\begin{align}
\mathcal B_m^{(\nu)}
&\ge-C\tau C_{2(\nu-1)/(2\nu-3)}^{\,2\nu-3}
\sum_{n=0}^{m-1}\|V_h^n\|_{\delta}^2 \label{eq:am34-bulk-lower-bound-solution}\\
&\quad-C\tau^2\sum_{n=n_0(\nu)}^{m}\|R^n\|_{\Omega}^2
-C\mathcal I_{\tau}(V_h) \label{eq:am34-bulk-lower-bound-forcing}
\end{align}
For AM5,
\begin{align}
\mathcal B_m^{(5)}
&\ge\left(\frac3{160}-Cc_0^2\right)
\sum_{n=4}^{m-1}\|d^3V_h^n\|_{\delta}^2 \label{eq:am5-bulk-lower-bound-solution}\\
&\quad-C\tau^2\sum_{n=n_0(5)}^{m}\|R^n\|_{\Omega}^2
-C\mathcal I_{\tau}(V_h) \label{eq:am5-bulk-lower-bound-forcing}
\end{align}
\end{lem}
\begin{proof}
For AM3 and AM4, combine the negative identities
\eqref{eq:am3-energy-bulk} and \eqref{eq:am4-energy-bulk} with
\eqref{eq:adams-high-increment-bulk-solution}--\eqref{eq:adams-high-increment-bulk-forcing}.
This gives a lower bound, not an upper bound. For AM5, substitute
\eqref{eq:am5-fourth-increment-control-solution}--\eqref{eq:am5-fourth-increment-control-forcing}
into \eqref{eq:am5-energy-bulk}.
\end{proof}

\begin{lem}[Boundary lower bound]
\label{lem:adams-boundary-lower-bound}
Let \(V_h\) solve \eqref{eq:am-general-forced-equation}. For every truncated
horizon \(J\), \(n_0(\nu)\le J\le N\), choose \(m_J\) so that
\begin{equation}
\|V_h^{m_J}\|_{\delta}
=\max_{n_0(\nu)\le n\le J}\|V_h^n\|_{\delta} \label{eq:adams-maximum-energy-index}
\end{equation}
For \(c_0\) sufficiently small,
\begin{align}
\mathrm{BT}_{m_J}^{(\nu)}
&\ge\left(\frac12-Cc_0^2\right)
\|V_h^{m_J}\|_{\delta}^2
-C\mathcal I_{\tau}(V_h) \label{eq:adams-boundary-lower-bound-solution}\\
&\quad-C\tau^2\sum_{n=n_0(\nu)}^{m_J}\|R^n\|_{\Omega}^2
\label{eq:adams-boundary-lower-bound-forcing}
\end{align}
\end{lem}
\begin{proof}
Apply Young's inequality to the endpoint cross terms in
\eqref{eq:am3-energy-boundary},
\eqref{eq:am4-energy-boundary-main}--\eqref{eq:am4-energy-boundary-increments},
and \eqref{eq:am5-energy-boundary-main}--\eqref{eq:am5-energy-boundary-increments}.
The previous-value recurrence gives
\eqref{eq:adams-endpoint-increment-solution}--\eqref{eq:am5-previous-endpoint-increment}.
It bounds every final increment occurring in the boundary formulas by the
maximum graph energy, the full squared forcing contribution displayed in
\eqref{eq:adams-boundary-lower-bound-forcing}, and start-up data. The terms at
levels \(1\), \(2\), or \(3\) have no fixed sign and are therefore bounded
below by \(-C\mathcal I_{\tau}(V_h)\). Young's inequality then absorbs the
remaining final-increment terms into the leading one-half of the final
energy for \(c_0\) sufficiently small.
\end{proof}


\subsection{Stability and Error Estimates}
\label{subsec:adams-stability-error}

\paragraph{Adams Method-Adapted Norm.}
The norm must reflect the bulk term available for each method. For the
AM\(\nu\) operator \(A_{\tau}\), use \(n_0=n_0(\nu)\) in
\(\mathcal I_{\tau}\), and define the common graph energy by
\begin{equation}
e_{\tau,\nu}^n(V)=\|T_{\delta}V^n\|_{\Omega}^2 \label{eq:adams-method-energy}
\end{equation}
For AM3 and AM4, the dissipation consists only of the material residual:
\begin{equation}
d_{\tau,\nu}^n(V)=\delta\|A_{\tau}V^n\|_{\Omega}^2
\quad \nu\in\{3,4\} \label{eq:am34-method-dissipation}
\end{equation}
For AM5, the positive third-increment bulk term supplies an additional
dissipation component, so we define
\begin{align}
d_{\tau,5}^4(V)&=\delta\|A_{\tau}V^4\|_{\Omega}^2 \label{eq:am5-initial-dissipation}\\
d_{\tau,5}^n(V)&=\delta\|A_{\tau}V^n\|_{\Omega}^2
+\tau^{-1}\|d^3V^{n-1}\|_{\delta}^2
\quad n\ge5 \label{eq:am5-method-dissipation}
\end{align}
Thus the corresponding sequence norms are exactly
\begin{align}
\|V\|_{\mathcal E_\nu,m}^2
&=\max_{n_0(\nu)\le n\le m}\|T_{\delta}V^n\|_{\Omega}^2
+\tau\delta\sum_{n=n_0(\nu)}^m\|A_{\tau}V^n\|_{\Omega}^2
\quad \nu\in\{3,4\} \label{eq:am34-method-adapted-norm}\\
\|V\|_{\mathcal E_5,m}^2
&=\max_{4\le n\le m}\|T_{\delta}V^n\|_{\Omega}^2
+\tau\delta\sum_{n=4}^m\|A_{\tau}V^n\|_{\Omega}^2
+\sum_{n=4}^{m-1}\|d^3V^n\|_{\delta}^2 \label{eq:am5-method-adapted-norm}
\end{align}
With these method-dependent choices understood, define the corresponding
lifted dual norm by
\begin{equation}
\|R\|_{\mathcal E_\nu^*,N}
=\max_{n_0(\nu)\le m\le N}
\sup_{\substack{V^n=0\text{ for }n\notin\{0,\ldots,m\}\\
\mathcal I_{\tau}(V)+\|V\|_{\mathcal E_\nu,m}^2>0}}
\frac{\tau\sum_{n=n_0(\nu)}^m(R^n,P_{\delta}V^n)_{\Omega}}
{\bigl(\mathcal I_{\tau}(V)+\|V\|_{\mathcal E_\nu,m}^2\bigr)^{1/2}}
\label{eq:adams-lifted-dual-norm}
\end{equation}
The finite start-up stencil is included in the denominator, and the test
sequence is zero-extended after its horizon before \(P_{\delta}\) is
applied. The shift in \eqref{eq:am5-method-dissipation} makes the last term of
\eqref{eq:am5-method-adapted-norm} agree exactly with the positive AM5 bulk
sum in \eqref{eq:am5-energy-bulk}.

\begin{thm}[Adams--Moulton stability]
\label{thm:adams-moulton-stability}
Assume the inverse inequality \eqref{eq:adams-inverse-inequality} and the CFL
condition \eqref{eq:cflcond}, with \(c_0\) sufficiently small. For AM3 also
assume \eqref{eq:am3-strengthened-cfl}, and for AM4 assume
\eqref{eq:am4-strengthened-cfl}. If \(V_h\) solves
\eqref{eq:am-general-forced-equation}, then
\begin{equation}
\|V_h\|_{\mathcal E_\nu,N}^2
\lesssim\mathcal I_{\tau}(V_h)
+\|R\|_{\mathcal E_\nu^*,N}^2
+\tau\|R\|_{\ell_{\tau}^2(n_0(\nu),N;\mathcal H)}^2
\label{eq:adams-stability-right-hand-side}
\end{equation}
For AM3 and AM4, the hidden constant depends exponentially on
\(TC_{2(\nu-1)/(2\nu-3)}^{\,2\nu-3}\). For AM5 it is independent of a
strengthened CFL constant.
The theorem requires only \(0<\delta\le C_{\delta}h\); hence all three
conclusions include the symmetric positive definite normal-equation choice
\(\delta=b_0\tau\) under
their stated CFL conditions.
\end{thm}
\begin{proof}
The proof has two distinct branches. For AM3 and AM4, the negative bulk term
is controlled by the increment recurrence and then by a discrete Gronwall
argument. For AM5, the positive third-increment term absorbs the negative
fourth-increment term, so no Gronwall step is needed.
Use \(S_{\delta}V_h^n\) as the test function in
\eqref{eq:am-general-forced-equation}, sum to an arbitrary horizon, and
insert Lemma \ref{lem:adams-exact-energy-decompositions}. Write
\begin{equation}
K_\nu=C_{2(\nu-1)/(2\nu-3)}^{\,2\nu-3},\qquad
M_J=\max_{n_0(\nu)\le n\le J}\|V_h^n\|_{\delta}^2
\label{eq:am-gronwall-quantities}
\end{equation}
For each \(J\), choose the running maximum index \(m_J\) from
\eqref{eq:adams-maximum-energy-index} and apply the energy identity on the
truncated interval ending at \(m_J\). For AM3 and AM4, Lemmas
\ref{lem:adams-bulk-lower-bounds} and
\ref{lem:adams-boundary-lower-bound}, together with the dual-norm definition
and Young's inequality, give
\begin{align}
M_J
&\le C\mathcal A_N
+C\tau K_\nu\sum_{n=n_0(\nu)}^{J-1}M_n
\label{eq:am-running-maximum-recursion}\\
\mathcal A_N
&=\mathcal I_{\tau}(V_h)+\|R\|_{\mathcal E_\nu^*,N}^2
+\tau\|R\|_{\ell_{\tau}^2(n_0(\nu),N;\mathcal H)}^2
\label{eq:am-stability-data-quantity}
\end{align}
Crucially, the bulk estimate contains
\(\sum_{n=0}^{m_J-1}\|V_h^n\|_{\delta}^2\), not the current value. Its
start-up part is controlled by \(\mathcal I_{\tau}(V_h)\), and its computed
part is bounded by \(\sum_{n=n_0(\nu)}^{J-1}M_n\). The dual forcing term is
bounded by
\begin{equation}
\|R\|_{\mathcal E_\nu^*,N}
\bigl(\mathcal I_{\tau}(V_h)
+\|V_h\|_{\mathcal E_\nu,m_J}^2\bigr)^{1/2}
\label{eq:am-dual-forcing-at-running-maximum}
\end{equation}
so Young's inequality absorbs the running graph energy and the partial
material-residual sum. The endpoint and bulk forcing defects are exactly
\(C\tau^2\sum_{n=n_0(\nu)}^{m_J}\|R^n\|_{\Omega}^2\), which is bounded by the
last term of \(\mathcal A_N\).

The standard discrete Gronwall lemma applied to
\eqref{eq:am-running-maximum-recursion} yields
\begin{equation}
M_N\le C\mathcal A_N\exp(C T K_\nu)
\label{eq:am-running-maximum-gronwall}
\end{equation}
which displays the claimed dependence on the strengthened CFL constant.
Applying the energy identity once more at the final horizon \(N\), using
\eqref{eq:am-running-maximum-gronwall} for all preceding solution terms, and
again applying Young's inequality controls the complete
\(\tau\delta\sum_{n=n_0(\nu)}^N\|A_{\tau}V_h^n\|_{\Omega}^2\) sum.

For AM5, choose \(c_0\) so that the negative fourth-increment contribution
in Lemma \ref{lem:adams-bulk-lower-bounds} is absorbed by the positive
third-increment sum. The same running-maximum and dual-norm argument then
controls the graph energy without Gronwall. Repeating it at \(N\) controls
the full material-residual and third-increment sums. These arguments prove
\eqref{eq:adams-stability-right-hand-side} for the general forcing \(R\).
\end{proof}

\begin{cor}[Piecewise affine stability under hyperbolic CFL]
\label{cor:adams-piecewise-affine-stability}
Assume the piecewise affine mapping property
\eqref{eq:adams-piecewise-affine-mapping} and the parameter conditions
\eqref{eq:adams-piecewise-affine-parameters}, with \(c_0\) sufficiently
small. Then the AM3 and AM4 solutions satisfy
\eqref{eq:adams-stability-right-hand-side} without the strengthened
conditions \eqref{eq:am3-strengthened-cfl}--
\eqref{eq:am4-strengthened-cfl}. In particular, this includes the standard
hyperbolic regime \(\tau\lesssim\delta\lesssim h\), subject to the stated
smallness of \(\tau/h\).
\end{cor}
\begin{proof}
At an arbitrary horizon, substitute Lemma
\ref{lem:adams-piecewise-affine-increment-bound} into the negative bulk
terms \eqref{eq:am3-energy-bulk} and \eqref{eq:am4-energy-bulk}. This gives
\begin{align}
\mathcal B_m^{(\nu)}
&\ge-C C_{\tau\delta}c_0^2\tau\delta
\sum_{n=n_0(\nu)}^m\|A_{\tau}V_h^n\|_{\Omega}^2
\label{eq:adams-piecewise-affine-bulk-lower-bound-solution}\\
&\quad-C\tau^2\sum_{n=n_0(\nu)}^m\|R^n\|_{\Omega}^2
-C\mathcal I_{\tau}(V_h)
\label{eq:adams-piecewise-affine-bulk-lower-bound-forcing}
\end{align}
Choose \(c_0\) so that the first term is absorbed by the material-residual
term in \eqref{eq:adams-tested-energy-identity}. Lemma
\ref{lem:adams-boundary-lower-bound} controls the endpoint terms under the
same hyperbolic CFL condition. The maximum-horizon and final-horizon
arguments in the proof of Theorem \ref{thm:adams-moulton-stability} now give
\eqref{eq:adams-stability-right-hand-side}; no Gronwall term and no
strengthened CFL condition are needed.
\end{proof}

For the original inhomogeneous AM\(\nu\) scheme, take
\(R^n=\mathcal M_{\mathrm{AM}\nu}F^n\) in
\eqref{eq:adams-stability-right-hand-side}. Thus its right-hand side is
\begin{equation}
\mathcal I_{\tau}(U_h)
+\|\mathcal M_{\mathrm{AM}\nu}F\|_{\mathcal E_\nu^*,N}^2
+\tau\|\mathcal M_{\mathrm{AM}\nu}F\|_{
\ell_{\tau}^2(n_0(\nu),N;\mathcal H)}^2
\label{eq:adams-physical-stability-right-hand-side}
\end{equation}

\begin{cor}[Adams--Moulton error estimate]
\label{cor:adams-moulton-error}
Assume either the stability hypotheses of Theorem
\ref{thm:adams-moulton-stability} or, for AM3--AM4, the piecewise affine
hypotheses of Corollary \ref{cor:adams-piecewise-affine-stability}. Also
assume the regularity hypotheses of Proposition
\ref{prop:dual-residual-bound} and projected start-up data. Then
\eqref{eq:abstract-error-estimate-space}--\eqref{eq:abstract-error-estimate-time}
holds with \(\nu=3,4,5\) for AM3, AM4, and AM5, respectively.
\end{cor}
\begin{proof}
Apply the stability estimate from Theorem
\ref{thm:adams-moulton-stability} or Corollary
\ref{cor:adams-piecewise-affine-stability} to the error equation
\eqref{eq:error-equation} with general forcing
\(R=A_{\tau}\eta-\rho\). Use the residual identity
\eqref{eq:discrete-error-residual-splitting} and
\eqref{eq:adams-truncation-bound}--
\eqref{eq:am-differentiated-consistency-derivative}. Proposition
\ref{prop:dual-residual-bound} applies to the lifted norm
\eqref{eq:adams-lifted-dual-norm}, since its denominator contains the graph
maximum and the material-residual sum used in the proposition. Theorem
\ref{thm:meta-half} then gives
\eqref{eq:abstract-error-estimate-space}--\eqref{eq:abstract-error-estimate-time}
with the AM-specific norm \eqref{eq:am34-method-adapted-norm} or
\eqref{eq:am5-method-adapted-norm}.
\end{proof}
Under the balanced choice \(\delta\sim h\) and \(\tau\lesssim h\), this yields
the concrete rate \(O(h^{k+1/2}+\tau^\nu)\), under the relevant general-degree
stability condition or the piecewise affine relaxation above.

\begin{rem}[Scope of the piecewise affine relaxation]
\label{rem:adams-piecewise-affine-relaxation}
Corollary \ref{cor:adams-piecewise-affine-stability} depends essentially on
affine elements and constant-coefficient first-order operators: it uses that
the differentiated finite element function is elementwise constant.
The abstract inverse inequality alone does not provide this cancellation at
higher polynomial degree, for which the strengthened conditions therefore
remain in the present proof. This distinction is consistent with the
imaginary-axis stability picture for Adams methods \cite{GFR15}; the
strengthened exponents for skew-symmetric discretizations are discussed in
\cite{Der12}. Here the
streamline-diffusion term supplies the additional control needed for
\(\mathbb P_1\) under the standard hyperbolic CFL condition.
\end{rem}


\section{Explicit Schemes: Adams--Bashforth Methods}
\label{sec:ab}
This section analyses AB3 and AB4 directly from their coefficient expansions.
Because the current coefficient is \(b_0=0\), the methods are explicit and
cannot be residual-minimisation normal equations. We first establish
consistency, then derive exact energy decompositions and the increment bounds
needed to control their negative bulk terms. This leads to stability and error
estimates under the standard hyperbolic restriction \(\tau\lesssim h\) and the
upper bound \(0<\delta\lesssim h\). The formulation remains nonsymmetric for
every admissible \(\delta\), since \(b_0=0\).

\paragraph{Coefficients and Increment Expansions.}
The first computed indices are \(n_{\mathrm{AB},0}(3)=3\) and
\(n_{\mathrm{AB},0}(4)=4\). The method
averages in \eqref{eq:method-average} have coefficients
\begin{align}
\mathrm{AB3}:\quad(b_0,b_1,b_2,b_3)
&=\left(0,\frac{23}{12},-\frac{16}{12},\frac5{12}\right) \label{eq:ab3-coefficients}\\
\mathrm{AB4}:\quad(b_0,b_1,b_2,b_3,b_4)
&=\left(0,\frac{55}{24},-\frac{59}{24},
\frac{37}{24},-\frac9{24}\right) \label{eq:ab4-coefficients}
\end{align}
Using the backward increments in
\eqref{eq:first-backward-increment}--\eqref{eq:higher-backward-increments},
exact coefficient expansion gives
\begin{align}
\mathcal M_{\mathrm{AB3}}U^n
&=\widetilde U^n-\frac1{12}d^2U^n-\frac5{12}d^3U^n \label{eq:ab3-increment-average}\\
\mathcal M_{\mathrm{AB4}}U^n
&=\widetilde U^n-\frac1{12}d^2U^n-\frac1{24}d^3U^n
-\frac38d^4U^n \label{eq:ab4-increment-average}
\end{align}
The relatively large coefficient \(-3/8\) of the fourth increment in AB4
generates a negative fourth-increment bulk term, balanced by a positive
third-increment term in the energy decomposition below.


\subsection{Consistency of the Adams--Bashforth Methods}
\label{subsec:adams-bashforth-consistency}
Set \(\sigma(t)=F(t)-GU(t)=\partial_tU(t)\). For AB\(\nu\), let \(\sigma_{\nu-1}\)
interpolate this exact time derivative at the past nodes
\begin{equation}
t^{n-j}\quad 1\le j\le \nu \label{eq:ab-interpolation-nodes}
\end{equation}
Unlike the Adams--Moulton rule, this interpolant uses no value at the current
time. The Adams--Bashforth weights are the integrals of the corresponding
extrapolating Lagrange basis over the current step. Hence,
\begin{equation}
\int_{t^{n-1}}^{t^n}\sigma_{\nu-1}(t)\,dt
=\tau\mathcal M_{\mathrm{AB}\nu}\sigma(t^n) \label{eq:ab-interpolatory-quadrature}
\end{equation}
Subtracting this relation from \eqref{eq:adams-exact-integral} gives
\begin{align}
\rho^n
&=D_{\tau}^{-}U(t^n)+G\mathcal M_{\mathrm{AB}\nu}U(t^n)
-\mathcal M_{\mathrm{AB}\nu}F(t^n) \label{eq:ab-truncation-residual}\\
\|\rho^n\|_{\Omega}
&\le C_\nu\tau^{\nu-1/2}
\|\partial_t^\nu(F-GU)\|_{L^2(t^{n-\nu},t^n;\mathcal H)}
\quad \nu\in\{3,4\} \label{eq:ab-truncation-bound}
\end{align}
Thus AB3 and AB4 have consistency orders three and four, respectively. In
the notation of \eqref{eq:global-truncation-regularity}, take
\begin{equation}
\mathcal C_{n,\nu+1}(U,F)
=C_\nu\|\partial_t^\nu(F-GU)\|_{L^2(t^{n-\nu},t^n;\mathcal H)}
\label{eq:ab-local-truncation-regularity}
\end{equation}
The intervals overlap at most \(\nu+1\) times, so this choice gives
\eqref{eq:global-truncation-bound}. As in the Adams--Moulton analysis, the
dual-residual estimate also needs uniform control of the residual and
square-summable control of its first forward difference. The integral
remainder gives both bounds:
\begin{align}
\max_{\nu\le n\le N}\|\rho^n\|_{\Omega}
+\|D_{\tau}^{+}\rho\|_{\ell_{\tau}^2(\nu,N-1;\mathcal H)}
&\le C_\nu\tau^\nu\Bigl(
\|\partial_t^\nu(F-GU)\|_{L^{\infty}(0,T;\mathcal H)}
\label{eq:ab-differentiated-consistency-main}\\
&\quad+\|\partial_t^{\nu+1}(F-GU)\|_{L^2(0,T;\mathcal H)}\Bigr)
\label{eq:ab-differentiated-consistency-derivative}
\end{align}
Consequently, \eqref{eq:maximum-truncation-bound}--
\eqref{eq:differentiated-truncation-bound} hold for AB3 and AB4 with the
displayed norms included in \(\mathcal C^{\sharp}_{T,\nu+2}(U,F)\).


\subsection{Exact AB3 and AB4 Energy Decompositions}
\label{subsec:adams-bashforth-energy}
Assume the inverse inequality \eqref{eq:adams-inverse-inequality} and
\begin{equation}
0<\delta\le C_{\delta}h,\qquad
\tau\le c_{\mathrm{AB}}C_G^{-1}h \label{eq:ab-standard-cfl}
\end{equation}
where \(c_{\mathrm{AB}}>0\) is sufficiently small. No lower bound proportional
to \(h\) is imposed on \(\delta\). For \(\nu\in\{3,4\}\), define
\begin{equation}
\mathcal E_{\mathrm{AB}\nu,m}
=\sum_{n=\nu}^m(d^1U^n,\mathcal M_{\mathrm{AB}\nu}U^n)_{\delta}
\label{eq:ab-discrete-energy}
\end{equation}
The common tested identity gives
\begin{equation}
\tau\sum_{n=\nu}^m(A_{\tau}U^n,P_{\delta}U^n)_{\Omega}
=\mathcal E_{\mathrm{AB}\nu,m}
+\tau\delta\sum_{n=\nu}^m\|A_{\tau}U^n\|_{\Omega}^2
\label{eq:ab-tested-energy-identity}
\end{equation}

For both methods, the decomposition below pairs a positive bulk term in
\(d^{\nu-1}U\) with a negative bulk term in \(d^{\nu}U\).
The increment estimate in Section \ref{subsec:adams-bashforth-increments}
will show that the positive term absorbs the negative one under
\eqref{eq:ab-standard-cfl}.

\begin{lem}[Exact Adams--Bashforth decompositions]
\label{lem:ab-exact-energy-decompositions}
For every admissible horizon \(m\),
\begin{equation}
\mathcal E_{\mathrm{AB}\nu,m}
=\mathcal B_{\mathrm{AB}\nu,m}+\mathrm{BT}_{\mathrm{AB}\nu,m}
\quad \nu\in\{3,4\} \label{eq:ab-bulk-boundary-decomposition}
\end{equation}
For AB3, the bulk term is
\begin{equation}
\mathcal B_{\mathrm{AB3},m}
=\frac38\sum_{n=3}^{m-1}\|d^2U^n\|_{\delta}^2
-\frac5{24}\sum_{n=3}^{m}\|d^3U^n\|_{\delta}^2
\label{eq:ab3-energy-bulk}
\end{equation}
and the boundary term is
\begin{align}
\mathrm{BT}_{\mathrm{AB3},m}
&=\frac12\bigl(\|U^m\|_{\delta}^2-\|U^2\|_{\delta}^2\bigr)
-\frac1{24}\bigl(\|d^1U^m\|_{\delta}^2
-\|d^1U^2\|_{\delta}^2\bigr) \label{eq:ab3-energy-boundary-main}\\
&\quad-\frac5{12}\bigl((d^1U^m,d^2U^m)_{\delta}
-(d^1U^2,d^2U^2)_{\delta}\bigr)
+\frac16\|d^2U^m\|_{\delta}^2
+\frac5{24}\|d^2U^2\|_{\delta}^2
\label{eq:ab3-energy-boundary-increments}
\end{align}
For AB4, the bulk term is
\begin{equation}
\mathcal B_{\mathrm{AB4},m}
=\frac{13}{24}\sum_{n=4}^{m-1}\|d^3U^n\|_{\delta}^2
-\frac3{16}\sum_{n=4}^{m}\|d^4U^n\|_{\delta}^2
\label{eq:ab4-energy-bulk}
\end{equation}
and the boundary term is
\begin{align}
\mathrm{BT}_{\mathrm{AB4},m}
&=\frac12\bigl(\|U^m\|_{\delta}^2-\|U^3\|_{\delta}^2\bigr)
-\frac1{24}\bigl(\|d^1U^m\|_{\delta}^2
-\|d^1U^3\|_{\delta}^2\bigr) \label{eq:ab4-energy-boundary-main}\\
&\quad-\frac1{24}\bigl((d^1U^m,d^2U^m)_{\delta}
-(d^1U^3,d^2U^3)_{\delta}\bigr) \label{eq:ab4-energy-boundary-second}\\
&\quad-\frac38\bigl((d^1U^m,d^3U^m)_{\delta}
-(d^1U^3,d^3U^3)_{\delta}\bigr)
-\frac1{48}\|d^2U^m\|_{\delta}^2
+\frac1{48}\|d^2U^3\|_{\delta}^2
\label{eq:ab4-energy-boundary-third}\\
&\quad+\frac3{16}\|d^2U^{m-1}\|_{\delta}^2
-\frac3{16}\|d^2U^2\|_{\delta}^2
+\frac16\|d^3U^m\|_{\delta}^2
+\frac38\|d^3U^3\|_{\delta}^2
\label{eq:ab4-energy-boundary-fourth}
\end{align}
\end{lem}
\begin{proof}
Insert \eqref{eq:ab3-increment-average} and
\eqref{eq:ab4-increment-average} into \eqref{eq:ab-discrete-energy}. Apply
the summation identities
\eqref{eq:du-d2u}--\eqref{eq:adams-sbp-fourth-bulk} with the AB3 lower index
\(3\) and the AB4 lower index \(4\). For AB3, the interior coefficients are
\begin{align}
\frac5{12}-\frac1{24}&=\frac38 \label{eq:ab3-second-increment-coefficient}\\
-\frac12\frac5{12}&=-\frac5{24} \label{eq:ab3-third-increment-coefficient}
\end{align}
For AB4, the interior second-increment coefficient cancels, while
\begin{align}
\frac1{24}-\frac12\frac1{12}&=0 \label{eq:ab4-second-increment-cancellation}\\
\frac32\frac38-\frac12\frac1{24}&=\frac{13}{24}
\label{eq:ab4-third-increment-coefficient}\\
-\frac12\frac38&=-\frac3{16} \label{eq:ab4-fourth-increment-coefficient}
\end{align}
Collecting the endpoints gives the boundary terms stated above.
\end{proof}


\subsection{Increment and Boundary Control}
\label{subsec:adams-bashforth-increments}

The explicit relation contains no current spatial term. After taking time
differences, the highest increment is therefore controlled directly by the
preceding increment, with the small factor \(\tau/h\).

\begin{lem}[Explicit increment control]
\label{lem:ab-increment-control}
Let \(\nu\in\{3,4\}\), and let \(V_h\) solve the generally forced AB\(\nu\)
equation
\begin{equation}
(A_{\tau}V_h^n,T_{\delta}W)_{\Omega}
=(R^n,T_{\delta}W)_{\Omega}
\quad\forall W\in\mathcal V_h,\qquad n\ge \nu
\label{eq:ab-general-forced-equation}
\end{equation}
Then
\begin{align}
\sum_{n=\nu}^{m}\|d^{\nu}V_h^n\|_{\delta}^2
&\lesssim c_{\mathrm{AB}}^2
\sum_{n=\nu}^{m-1}\|d^{\nu-1}V_h^n\|_{\delta}^2
+\tau^2\sum_{n=\nu}^{m}\|R^n\|_{\Omega}^2
\label{eq:ab-high-increment-control-solution}\\
&\quad+C\mathcal I_{\tau}(V_h)
\label{eq:ab-high-increment-control-startup}
\end{align}
At every horizon \(m\ge \nu\),
\begin{align}
\sum_{j=1}^{\nu-1}\|d^jV_h^m\|_{\delta}^2
+\|d^{\nu-2}V_h^{m-1}\|_{\delta}^2
&\lesssim c_{\mathrm{AB}}^2
\max_{0\le n\le m}\|V_h^n\|_{\delta}^2
+\tau^2\sum_{n=\nu}^{m}\|R^n\|_{\Omega}^2
\label{eq:ab-endpoint-increment-control-solution}\\
&\quad+C\mathcal I_{\tau}(V_h)
\label{eq:ab-endpoint-increment-control-startup}
\end{align}
\end{lem}
\begin{proof}
For \(n\ge2\nu-1\), apply \(d^{\nu-1}\) to
\eqref{eq:ab-general-forced-equation}. Since the average has no current
value, the exact differenced equation is
\begin{align}
(d^{\nu}V_h^n,T_{\delta}W)_{\Omega}
&=\tau(d^{\nu-1}R^n,T_{\delta}W)_{\Omega}
\label{eq:ab-differenced-equation-forcing}\\
&\quad-\tau(G\mathcal M_{\mathrm{AB}\nu}d^{\nu-1}V_h^n,
T_{\delta}W)_{\Omega}\quad \forall W\in\mathcal V_h
\label{eq:ab-differenced-equation-solution}
\end{align}
For AB4, the history term thus contains the previous third differences
\(d^3V_h^{n-j}\). Choose \(W=d^{\nu}V_h^n\). Then
skew-symmetry, \eqref{eq:adams-graph-l2-equivalence}, and the fact that
\(b_0=0\) give
\begin{align}
\|d^{\nu}V_h^n\|_{\delta}
&\lesssim\frac{\tau}{h}\sum_{j=1}^{\nu}
\|d^{\nu-1}V_h^{n-j}\|_{\delta} \label{eq:ab-pointwise-increment-solution}\\
&\quad+\tau\sum_{q=0}^{\nu-1}
\|R^{n-q}\|_{\Omega}
\label{eq:ab-pointwise-increment-forcing}
\end{align}
For \(n\ge2\nu-1\), every equation used in the difference is a computed AB
equation. Squaring and summing proves the asserted bound over these indices.
The finitely many indices \(\nu\le n<2\nu-1\) follow by applying the explicit
update successively; their solution terms are controlled by
\(\mathcal I_{\tau}(V_h)\) and their forcing terms by the displayed forcing
sum. The endpoint estimate follows from the same previous-value recurrence
applied successively for orders \(1\) through \(\nu-1\).
\end{proof}

\begin{lem}[AB bulk and boundary lower bounds]
\label{lem:ab-bulk-boundary-lower-bounds}
Let \(V_h\) solve \eqref{eq:ab-general-forced-equation}, and evaluate the
bulk and boundary terms on this sequence. Under
\eqref{eq:ab-standard-cfl},
\begin{align}
\mathcal B_{\mathrm{AB3},m}
&\ge\left(\frac38-Cc_{\mathrm{AB}}^2\right)
\sum_{n=3}^{m-1}\|d^2V_h^n\|_{\delta}^2
-C\tau^2\sum_{n=3}^{m}\|R^n\|_{\Omega}^2
-C\mathcal I_{\tau}(V_h) \label{eq:ab3-bulk-lower-bound}\\
\mathcal B_{\mathrm{AB4},m}
&\ge\left(\frac{13}{24}-Cc_{\mathrm{AB}}^2\right)
\sum_{n=4}^{m-1}\|d^3V_h^n\|_{\delta}^2
-C\tau^2\sum_{n=4}^{m}\|R^n\|_{\Omega}^2
-C\mathcal I_{\tau}(V_h) \label{eq:ab4-bulk-lower-bound}
\end{align}
For \(\nu\in\{3,4\}\) and every truncated horizon \(J\), \(\nu\le J\le N\),
choose \(m_*\in\{\nu,\ldots,J\}\) so that
\(\|V_h^{m_*}\|_{\delta}=\max_{\nu\le n\le J}\|V_h^n\|_{\delta}\). Then
\begin{align}
\mathrm{BT}_{\mathrm{AB}\nu,m_*}
&\ge\left(\frac12-Cc_{\mathrm{AB}}^2\right)
\|V_h^{m_*}\|_{\delta}^2
-C\mathcal I_{\tau}(V_h) \label{eq:ab-boundary-lower-bound-solution}\\
&\quad-C\tau^2\sum_{n=\nu}^{m_*}\|R^n\|_{\Omega}^2
\quad \nu\in\{3,4\} \label{eq:ab-boundary-lower-bound-forcing}
\end{align}
\end{lem}
\begin{proof}
Substitute \eqref{eq:ab-high-increment-control-solution}--\eqref{eq:ab-high-increment-control-startup}
into the negative highest-increment sums in
\eqref{eq:ab3-energy-bulk} and \eqref{eq:ab4-energy-bulk}. This proves the
bulk bounds.

For the boundary terms, apply Young's inequality to every final cross term
in \eqref{eq:ab3-energy-boundary-main}--\eqref{eq:ab4-energy-boundary-fourth}
and use \eqref{eq:ab-endpoint-increment-control-solution}--\eqref{eq:ab-endpoint-increment-control-startup}.
All start-up contributions in the exact formulas, regardless of their sign,
are bounded below by \(-C\mathcal I_{\tau}(V_h)\). Forcing contributions from
the endpoint increments enter the lower bound with the negative sign and the
squared \(\tau^2\) scaling shown in
\eqref{eq:ab-boundary-lower-bound-forcing}. The final increment terms are
absorbed into the leading final energy for \(c_{\mathrm{AB}}\) sufficiently
small.
\end{proof}


\subsection{Stability and Error Estimates}
\label{subsec:adams-bashforth-stability}

\paragraph{AB Method-Adapted Norm.}
The positive bulk term controls \(d^{\nu-1}V\), which must therefore be included
in the method-adapted norm. For \(\nu\in\{3,4\}\), use \(n_0=\nu\) in
\(\mathcal I_{\tau}\), and define
\begin{align}
e_{\tau,\mathrm{AB}\nu}^n(V)
&=\|T_{\delta}V^n\|_{\Omega}^2 \label{eq:ab-method-energy}\\
d_{\tau,\mathrm{AB}\nu}^\nu(V)
&=\delta\|A_{\tau}V^\nu\|_{\Omega}^2 \label{eq:ab-initial-dissipation}\\
d_{\tau,\mathrm{AB}\nu}^n(V)
&=\delta\|A_{\tau}V^n\|_{\Omega}^2
+\tau^{-1}\|d^{\nu-1}V^{n-1}\|_{\delta}^2
\quad n\ge \nu+1 \label{eq:ab-method-dissipation}
\end{align}
The resulting sequence norm is
\begin{align}
\|V\|_{\mathcal E_{\mathrm{AB}\nu},m}^2
&=\max_{\nu\le n\le m}\|T_{\delta}V^n\|_{\Omega}^2
+\tau\delta\sum_{n=\nu}^m\|A_{\tau}V^n\|_{\Omega}^2 \label{eq:ab-method-adapted-norm-main}\\
&\quad+\sum_{n=\nu}^{m-1}\|d^{\nu-1}V^n\|_{\delta}^2
\label{eq:ab-method-adapted-norm-increment}
\end{align}
Define its lifted dual norm by
\begin{equation}
\|R\|_{\mathcal E_{\mathrm{AB}\nu}^*,N}
=\max_{\nu\le m\le N}
\sup_{\substack{V^n=0\text{ for }n\notin\{0,\ldots,m\}\\
\mathcal I_{\tau}(V)+\|V\|_{\mathcal E_{\mathrm{AB}\nu},m}^2>0}}
\frac{\tau\sum_{n=\nu}^m(R^n,P_{\delta}V^n)_{\Omega}}
{\bigl(\mathcal I_{\tau}(V)
+\|V\|_{\mathcal E_{\mathrm{AB}\nu},m}^2\bigr)^{1/2}}
\label{eq:ab-lifted-dual-norm}
\end{equation}
The denominator controls the finite start-up stencil, and the test sequence
is zero-extended after its horizon before \(P_{\delta}\) is applied.

\begin{thm}[Adams--Bashforth stability]
\label{thm:adams-bashforth-stability}
Assume \eqref{eq:adams-inverse-inequality} and
\eqref{eq:ab-standard-cfl}, with \(c_{\mathrm{AB}}\) sufficiently small.
For AB\(\nu\), \(\nu\in\{3,4\}\), let \(V_h\) solve
\eqref{eq:ab-general-forced-equation}. Then
\begin{equation}
\|V_h\|_{\mathcal E_{\mathrm{AB}\nu},N}^2
\lesssim\mathcal I_{\tau}(V_h)
+\|R\|_{\mathcal E_{\mathrm{AB}\nu}^*,N}^2
+\tau\|R\|_{\ell_{\tau}^2(\nu,N;\mathcal H)}^2
\label{eq:ab-stability-estimate}
\end{equation}
\end{thm}
\begin{proof}
The argument follows the same absorption mechanism as for AM5: the positive
lower-increment bulk term controls the negative highest-increment term, so no
discrete Gronwall step is required.
Use \eqref{eq:ab-tested-energy-identity} at an arbitrary horizon and insert
Lemma \ref{lem:ab-exact-energy-decompositions}. For each truncated horizon
\(J\), apply the identity at the running maximum index \(m_*\) from Lemma
\ref{lem:ab-bulk-boundary-lower-bounds}. The lemma then gives positive control
of the running graph-energy maximum and of the \(d^{\nu-1}\)-increment sum when
\(c_{\mathrm{AB}}\) is sufficiently small. The right-hand side is bounded by
\eqref{eq:ab-lifted-dual-norm} and Young's inequality; the start-up factor in
the dual denominator is absorbed by \(\mathcal I_{\tau}(V_h)\). All increment
and boundary forcing defects are bounded by
\begin{equation}
C\tau^2\sum_{n=\nu}^N\|R^n\|_{\Omega}^2
=C\tau\|R\|_{\ell_{\tau}^2(\nu,N;\mathcal H)}^2
\label{eq:ab-forcing-defect-scaling}
\end{equation}
Repeating the argument at the final horizon controls the complete
material-residual sum and
proves \eqref{eq:ab-stability-estimate}.
\end{proof}

For the original inhomogeneous AB\(\nu\) scheme, take
\(R^n=\mathcal M_{\mathrm{AB}\nu}F^n\) in
\eqref{eq:ab-stability-estimate}.

\begin{cor}[Adams--Bashforth error estimate]
\label{cor:adams-bashforth-error}
Assume Theorem \ref{thm:adams-bashforth-stability}, the regularity hypotheses
of Proposition \ref{prop:dual-residual-bound}, and projected start-up data.
Then
\eqref{eq:abstract-error-estimate-space}--\eqref{eq:abstract-error-estimate-time}
holds with \(\nu=3\) for AB3 and \(\nu=4\) for AB4.
\end{cor}
\begin{proof}
Apply Theorem \ref{thm:adams-bashforth-stability} to the error equation with
general forcing
\(R=A_{\tau}\eta-\rho\). Use \eqref{eq:ab-truncation-bound}--
\eqref{eq:ab-differentiated-consistency-derivative}. Proposition
\ref{prop:dual-residual-bound} applies to
\eqref{eq:ab-lifted-dual-norm}, whose denominator contains the graph maximum
and material-residual sum required by the proposition. Theorem
\ref{thm:meta-half} then gives the asserted estimate.
\end{proof}
Under the balanced choice \(\delta\sim h\) and \(\tau\lesssim h\), this yields
the concrete rate \(O(h^{k+1/2}+\tau^\nu)\), with \(\nu=3\) for AB3 and \(\nu=4\)
for AB4.

\begin{rem}[The AB2 scheme]
\label{rem:adams-bashforth-two}
The second-order Adams--Bashforth scheme (AB2) can be analysed by arguments
similar to those used above. For finite element spaces of arbitrary polynomial
degree, they yield stability under the \(4/3\)-CFL condition
\(\tau\lesssim h^{4/3}\). For continuous piecewise affine approximation, the
elementwise cancellation used in Corollary
\ref{cor:adams-piecewise-affine-stability} instead yields stability under the
standard hyperbolic CFL condition \(\tau\lesssim h\). In the latter case, the
second-order discretisation errors in space and time are balanced. For an
analysis of AB2 with continuous interior penalty stabilisation, see
\cite{BG22}.
\end{rem}


\section{Numerical Experiments}
\label{sec:numerical-experiments}
The experiments address three questions. First, do the time integrators attain
their formal orders when spatial error is removed? Second, are the predicted
fully discrete rates observed when temporal and spatial orders are matched?
Third, does stabilisation confine errors generated by discontinuities and
compactly supported waves? We use the method coefficients and parameter
conventions of Sections \ref{sec:theta-method}, \ref{sec:adams-moulton}, and
\ref{sec:ab}. Rates computed from successive refinements are empirical; they
neither determine universal stability thresholds nor establish a general
localisation theorem.


\subsection{Implementation and Test Problems}
\label{subsec:numerical-implementation}
We use the acoustic system \eqref{eq:acoustic-pressure}--
\eqref{eq:acoustic-spatial-operator} with unit wave speed. Thus the scaled
and physical velocities agree, and \(U=(p,u)\). The one-dimensional domain is
\((0,1)\), and the two-dimensional domain is \((0,1)^2\), in both cases with
periodic boundary conditions. All components are approximated in the same
continuous finite element space: periodic \(\mathbb P_1\) through \(\mathbb P_4\) spaces in one
dimension and periodic triangular \(\mathbb P_1\) spaces in two dimensions.

The assembly uses consistent mass matrices and Gaussian quadrature of
sufficient order. Exact finite element histories are used in the purely
temporal test, while the smooth PDE tests use \(L^2\)-projected exact
histories. The discontinuous one-dimensional data are also \(L^2\)-projected,
with the jumps aligned to element interfaces on every mesh. The compact
initial pressure in the two-dimensional localisation test is instead
interpolated at the nodes so that its nodal support is preserved. For each fixed run, one
direct sparse factorisation is reused at every time step. There is no mass
lumping, filtering, limiter, added viscosity, or retuning along a refinement
path.

For every production run, a nominal step \(\tau_*\) is prescribed and the
actual step is chosen as \(\tau=T/N\leq\tau_*\), with integer \(N\), so that
the final time is an exact grid level.

Let \(e^n=U_h^n-U(t^n)\). The three error quantities reported below are,
respectively, the final \(L^2\) error, the maximum graph error, and the
accumulated material-residual error,
\begin{align}
E_{L^2}&=\|e^N\|_{\Omega} \label{eq:numerical-final-l2-error}\\
E_{G}&=\max_{0\le n\le N}
\bigl(\|e^n\|_{\Omega}^2+\delta^2\|Ge^n\|_{\Omega}^2\bigr)^{1/2}
\label{eq:numerical-maximum-graph-error}\\
E_{R}&=\left(\tau\delta\sum_{n=n_0}^{N}
\|D_{\tau}^{-}e^n+G\mathcal M_{\tau}e^n\|_{\Omega}^2\right)^{1/2}
\label{eq:numerical-material-residual-error}
\end{align}
Here \(n_0\) is the first computed time level and \(\mathcal M_{\tau}\) is the
average associated with the active method. These definitions use the same
graph and material-residual scaling as the energy analysis. Matrix assembly
reproduced both the discrete skew-symmetric structure and, when
\(\delta=b_0\tau\), the normal-equation cancellation to
floating-point precision. Pilot calculations were used only to select
conservative production constants. Unstable pilot points were kept in
the reproducibility record, and the scans are not interpreted as sharp or
mesh-independent CFL bounds.


\subsection{Temporal Accuracy}
\label{subsec:numerical-temporal-accuracy}
To remove spatial approximation error, we work on a periodic \(\mathbb P_1\) mesh with
three elements and prescribe the nontrivial finite element vector \(W_h\) by
its nodal values
\begin{equation}
W_h(x_i)=
\begin{pmatrix}
\sin(2\pi x_i)+0.3\cos(2\pi x_i)\\
0.7\cos(2\pi x_i)-0.2\sin(2\pi x_i)
\end{pmatrix}
\label{eq:numerical-temporal-mode}
\end{equation}
The discrete manufactured data are
\begin{align}
U_h(t)&=\exp(t)W_h \label{eq:numerical-temporal-solution}\\
F_h(t)&=\partial_tU_h(t)+GU_h(t) \label{eq:numerical-temporal-forcing}
\end{align}
up to \(T=1\), with \(N=10,20,40,80,160\), and \(320\) time steps. For the
\(\theta\)-methods, \(\delta=\theta\tau\); for all Adams methods, \(\delta=h\).
Table~\ref{tab:numerical-temporal-orders} reports the mean of the three finest
pairwise rates unaffected by roundoff. Every method attains its formal order.
For all methods except AM5, these are the final three rates. AM5 reaches a
floating-point error floor on the last refinement, so that level remains in
the reproducibility record while its rate is excluded from the mean.

\begin{table}[htbp]
\centering
\begin{tabular}{lcc}
\toprule
Method & Formal Order & Observed Order \\
\midrule
\(\theta\)-method, \(\theta=1\) & 1 & 1.0034 \\
\(\theta\)-method, \(\theta=1/2\) & 2 & 1.9986 \\
AM3 & 3 & 2.9959 \\
AM4 & 4 & 3.9906 \\
AM5 & 5 & 4.9129 \\
AB3 & 3 & 2.9861 \\
AB4 & 4 & 3.9794 \\
\bottomrule
\end{tabular}
\caption{Formal temporal orders and mean observed rates over the three finest
time-step refinements unaffected by roundoff.}
\label{tab:numerical-temporal-orders}
\end{table}


\subsection{Order-Matched One-Dimensional Convergence}
\label{subsec:numerical-one-dimensional-convergence}
The fully discrete convergence test is the homogeneous periodic right-moving
wave
\begin{equation}
p(x,t)=u(x,t)=\sin\bigl(2\pi(x-t)\bigr),\qquad T=1
\label{eq:numerical-one-dimensional-wave}
\end{equation}
Under a hyperbolic scaling \(\tau\sim h\), a method of temporal order \(\nu\)
has temporal error \(O(h^\nu)\). We therefore pair it with continuous
\(\mathbb P_{\nu-1}\) elements, whose formal \(L^2\)-approximation order is
also \(\nu\),
so that spatial error does not mask the benefit of high-order time
integration. The four
method--space pairs are Crank--Nicolson/\(\mathbb P_1\), AM3/\(\mathbb P_2\),
AM4/\(\mathbb P_3\), and AM5/\(\mathbb P_4\).

The meshes have \(10,20,40,80\), and \(160\) elements. Every run uses the
nominal parameters
\begin{equation}
\tau_*=0.1h,\qquad \delta=b_0\tau
\label{eq:numerical-matched-order-parameters}
\end{equation}
Consequently, all four methods use the normal-equation formulation and a
Cholesky factorisation. The start-up histories are \(L^2\)-projections of
the exact solution. For AM3/\(\mathbb P_2\) and AM4/\(\mathbb P_3\), the standard
hyperbolic scaling in \eqref{eq:numerical-matched-order-parameters} is less
restrictive than the sufficient general-degree CFL conditions in Section
\ref{sec:adams-moulton}. Their stable behaviour here is therefore numerical
evidence, not an extension of the theorem.

Table~\ref{tab:numerical-one-dimensional-rates} shows the rates on the finest
mesh pair. The final \(L^2\) and maximum graph errors recover orders
\(2,3,4\), and \(5\), respectively. The accumulated material-residual rates
recover the corresponding half-order sequence
\(3/2,5/2,7/2\), and \(9/2\). This experiment therefore exhibits the
fully discrete high-order behaviour in a PDE problem, rather than removing
the spatial error as in the manufactured temporal test. Fixed-\(\mathbb P_1\)
comparisons were retained in the reproducibility record as spatial controls,
but are not used to assess the advantage of the higher-order time methods.

\begin{table}[htbp]
\centering
\begin{tabular}{lcccc}
\toprule
Method & Space & Final \(L^2\) & Max. Graph & Material Residual \\
\midrule
Crank--Nicolson & \(\mathbb P_1\) & 2.0032 & 2.0028 & 1.5010 \\
AM3 & \(\mathbb P_2\) & 2.9815 & 2.9816 & 2.4791 \\
AM4 & \(\mathbb P_3\) & 4.0008 & 4.0010 & 3.4994 \\
AM5 & \(\mathbb P_4\) & 4.9984 & 4.9965 & 4.4968 \\
\bottomrule
\end{tabular}
\caption{Observed finest-pair rates, from \(80\) to \(160\) elements, for
the order-matched one-dimensional travelling-wave experiment.}
\label{tab:numerical-one-dimensional-rates}
\end{table}


\subsection{Two-Dimensional Convergence}
\label{subsec:numerical-two-dimensional-convergence}
This test keeps the spatial degree fixed at \(\mathbb P_1\). It is therefore not
designed to distinguish the formal orders of the time integrators; its purpose
is to test the graph and material-residual rates in two dimensions and to
compare implicit and explicit schemes on the same spatial family. The exact
solution is the oblique periodic plane wave
\begin{align}
\boldsymbol{\kappa}&=(1,1) \label{eq:numerical-plane-wave-vector}\\
p(x,y,t)&=\sin\bigl(2\pi(x+y-\sqrt{2}t)\bigr)
\label{eq:numerical-plane-wave-pressure}\\
u(x,y,t)&=2^{-1/2}(1,1)p(x,y,t),\qquad T=2^{-1/2}
\label{eq:numerical-plane-wave-velocity}
\end{align}
which exercises both velocity components. Crank--Nicolson, AM3, AM5, and
AB4 are run on slash-diagonal triangular meshes with \(8,16,32\), and \(64\)
subdivisions in each coordinate direction. The nominal refinement paths are
\begin{center}
\begin{tabular}{lcc}
\toprule
Method & Nominal step \(\tau_*\) & Stabilisation \(\delta\) \\
\midrule
Crank--Nicolson & \(0.25h\) & \(\tau/2\) \\
AM3 & \(0.1h^{4/3}\) & \(h\) \\
AM5 & \(0.02h\) & \(h\) \\
AB4 & \(0.003h\) & \(h\) \\
\bottomrule
\end{tabular}
\end{center}

The finest-pair rates in Table~\ref{tab:numerical-two-dimensional-rates} and
the curves in Figure~\ref{fig:numerical-two-dimensional-convergence} show
graph rates near two and material-residual rates near \(3/2\). In the figure,
the AM3, AM5, and AB4 curves are summarised by their pointwise median and
range because the individual curves are indistinguishable at the plotted
scale. This reflects the dominance of the \(\mathbb P_1\) spatial error once the
temporal error is small. These are the principal comparisons with the error
analysis. The multistep final \(L^2\) rates near \(2.87\) are an observed
structured-mesh superconvergence effect and are not claimed as a general
\(\mathbb P_1\) rate.

\begin{table}[htbp]
\centering
\begin{tabular}{lccc}
\toprule
Method & Final \(L^2\) & Maximum Graph & Material Residual \\
\midrule
CN & 2.1433 & 2.0950 & 1.5045 \\
AM3 & 2.8689 & 2.1685 & 1.4929 \\
AM5 & 2.8681 & 2.1680 & 1.4925 \\
AB4 & 2.8704 & 2.1691 & 1.4936 \\
\bottomrule
\end{tabular}
\caption{Observed finest-pair rates for the two-dimensional oblique plane
wave on the slash-diagonal production meshes.}
\label{tab:numerical-two-dimensional-rates}
\end{table}

\begin{figure}[htbp]
\centering
\includegraphics[width=\textwidth]{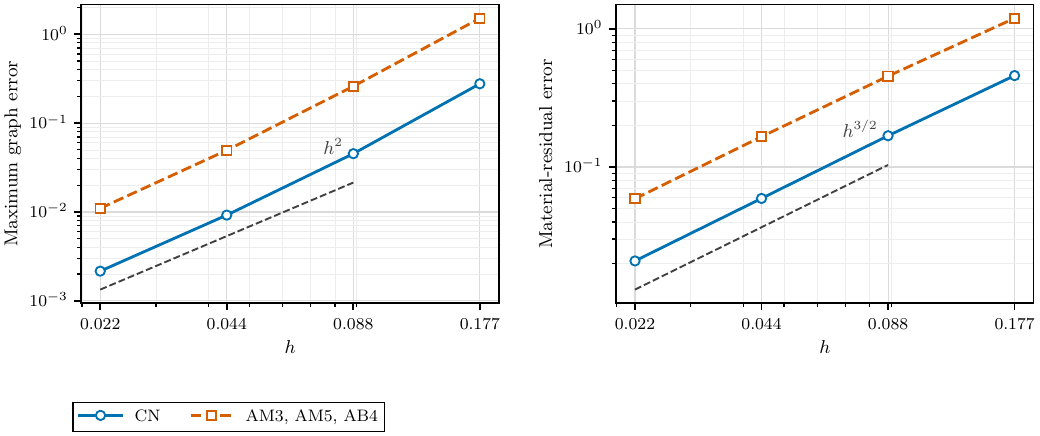}
\caption{Observed maximum graph errors and accumulated material-residual
errors for the two-dimensional oblique plane wave. The orange curve is the
pointwise median of the AM3, AM5, and AB4 errors, and the shaded band spans
their full range; Table~\ref{tab:numerical-two-dimensional-rates} gives the
individual finest-pair rates. Reference slopes indicate \(h^2\) and
\(h^{3/2}\), respectively.}
\label{fig:numerical-two-dimensional-convergence}
\end{figure}

The mesh-orientation check used slash, backslash, and alternating diagonals.
Across these orientations, the CN graph rates range from \(2.015\) to
\(2.095\), and its material-residual rates range from \(1.502\) to
\(1.506\). The AM5 graph rates range from \(2.056\) to \(2.168\), and its
material-residual rates range from \(1.493\) to \(1.498\). Error levels are
orientation-sensitive: the largest ratios between orientations are
approximately \(2.44\) for the graph error and \(3.48\) for the AM5 final
\(L^2\) error. The graph and material-residual rate trends nevertheless remain
robust across the tested orientations.


\subsection{One-Dimensional Localisation of Discontinuities}
\label{subsec:numerical-one-dimensional-localisation}
The preceding smooth tests do not address how discretisation error generated
near a discontinuity propagates through the domain. We therefore ask whether optimal
\(\mathbb P_1\) convergence is recovered in regions that remain a fixed physical
distance from the moving jumps. We prescribe zero initial velocity and the
periodic pressure profile
\begin{equation}
f(x)=
\begin{cases}
1+\dfrac14\sin\bigl(2\pi(x-x_L)/(x_R-x_L)\bigr),&x_L\le x<x_R\\
0,&\text{otherwise}
\end{cases}
\quad x_L=0.25,\quad x_R=0.45 \label{eq:numerical-discontinuous-profile}
\end{equation}
The profile is smooth and nonconstant in the interior of its support but has a
unit jump at each endpoint. For the one-dimensional acoustic system,
d'Alembert's formula gives the exact solution,
\begin{align}
p(x,t)&=\frac12\bigl(f(x-t)+f(x+t)\bigr)
\label{eq:numerical-dalembert-pressure}\\
u(x,t)&=\frac12\bigl(f(x-t)-f(x+t)\bigr)
\label{eq:numerical-dalembert-velocity}
\end{align}
Thus the initial pulse splits into left- and right-moving waves with a known
set of moving discontinuities,
\begin{equation}
\mathcal J(t)=\{x_L+t,x_R+t,x_L-t,x_R-t\}\bmod 1
\label{eq:numerical-moving-jump-set}
\end{equation}

We use periodic \(\mathbb P_1\) elements on meshes with \(40\) through \(2560\)
elements and Crank--Nicolson up to \(T=0.15\), with nominal step
\(\tau_*=0.2h\). The two calculations differ only in the stabilisation
parameter: \(\delta=0\) for the standard Galerkin method and
\(\delta=\tau/2\) for the normal-equation SUPG method. The endpoints \(x_L\)
and \(x_R\) are element interfaces on every mesh, and the initial pressure is
the \(L^2\)-projection of \eqref{eq:numerical-discontinuous-profile}.

For the fixed physical buffer width \(\rho_{\mathrm{buf}}=0.04\), define the
smooth region by
\begin{equation}
\Omega_{\mathrm{sm}}(t)
=\bigl\{x\in(0,1):\operatorname{dist}_{\mathrm{per}}
(x,\mathcal J(t))>\rho_{\mathrm{buf}}\bigr\} \label{eq:numerical-smooth-region}
\end{equation}
and measure the maximum local error,
\begin{equation}
E_{\mathrm{loc}}
=\max_{0\le n\le N}\|e^n\|_{\Omega_{\mathrm{sm}}(t^n)}
\label{eq:numerical-local-error}
\end{equation}
The same moving region is used at every refinement level, rather than a
mesh-dependent exclusion zone.

Table~\ref{tab:numerical-discontinuous-localisation} shows the finest-pair
results. The unstabilised local error converges at rate \(0.4995\), whereas
normal-equation SUPG recovers the optimal \(\mathbb P_1\) local rate \(2.0006\).
On the finest mesh their maximum local errors are \(5.6971\times10^{-3}\)
and \(3.9288\times10^{-7}\), respectively, a factor of approximately
\(1.45\times10^4\). In the part of the smooth region where neither exact
travelling pulse is present, the corresponding maximum-in-time \(L^2\) norms
of the numerical state are
\(2.9652\times10^{-3}\) and \(4.5629\times10^{-15}\). The full-domain errors
converge much more slowly because they include the moving discontinuities.
Their observed rates are reported without interpreting them as universal
asymptotic exponents.

\begin{table}[htbp]
\centering
\begin{tabular}{lcccc}
\toprule
Method & Global Error & Global Rate & Max. Local Error & Local Rate \\
\midrule
Unstabilised & \(2.5506\times10^{-2}\) & 0.3496 &
\(5.6971\times10^{-3}\) & 0.4995 \\
Normal SUPG & \(1.9694\times10^{-2}\) & 0.3750 &
\(3.9288\times10^{-7}\) & 2.0006 \\
\bottomrule
\end{tabular}
\caption{Errors on the \(2560\)-element mesh and rates from \(1280\) to
\(2560\) elements for the discontinuous d'Alembert experiment. The local
error is defined by \eqref{eq:numerical-local-error}.}
\label{tab:numerical-discontinuous-localisation}
\end{table}

The convergence curves in
Figure~\ref{fig:numerical-discontinuous-convergence} show that the stabilised
calculation passes through a preasymptotic layer-decay regime before settling
onto the \(h^2\) reference line. The standard Galerkin local error instead
follows the \(h^{1/2}\) line on the finest levels. The profiles in
Figure~\ref{fig:numerical-discontinuous-profiles} explain the contrast: the
unstabilised dispersive error extends across the periodic interval, whereas
the SUPG error remains concentrated around the exact wave fronts. This is
numerical evidence for localisation in this linear one-dimensional
configuration, not a general theorem for nonsmooth solutions.

\begin{figure}[htbp]
\centering
\includegraphics[width=\textwidth]{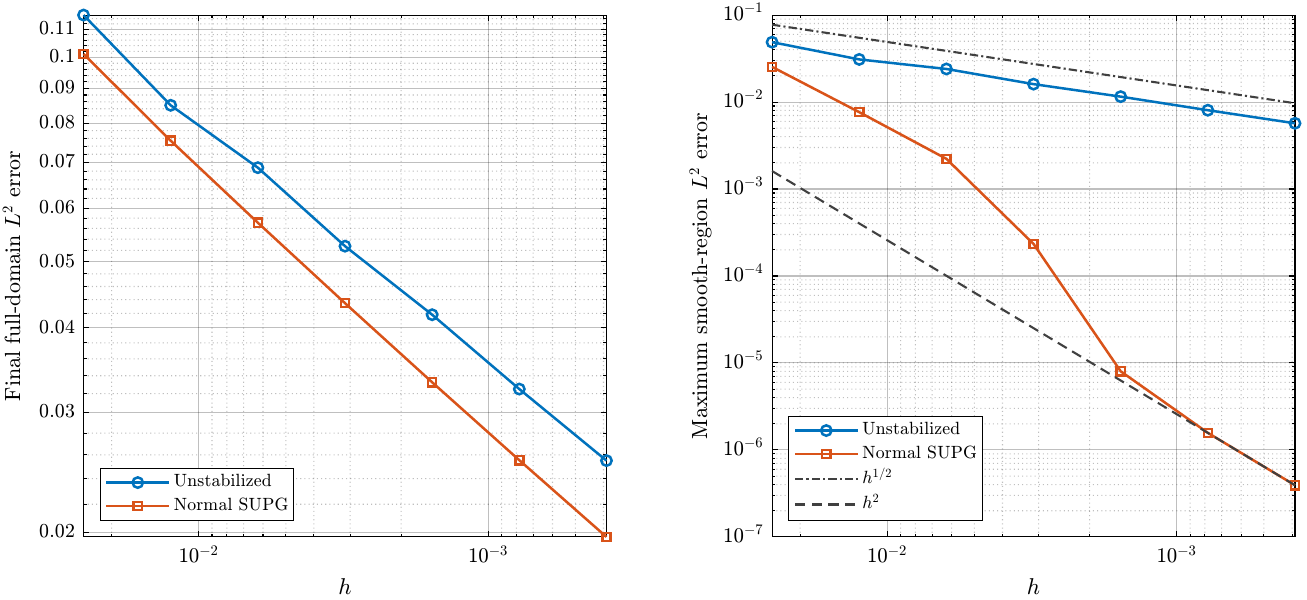}
\caption{Final full-domain errors and maximum smooth-region errors for the
discontinuous d'Alembert experiment. The fixed smooth region is given by
\eqref{eq:numerical-smooth-region}.}
\label{fig:numerical-discontinuous-convergence}
\end{figure}

\begin{figure}[htbp]
\centering
\includegraphics[width=\textwidth]{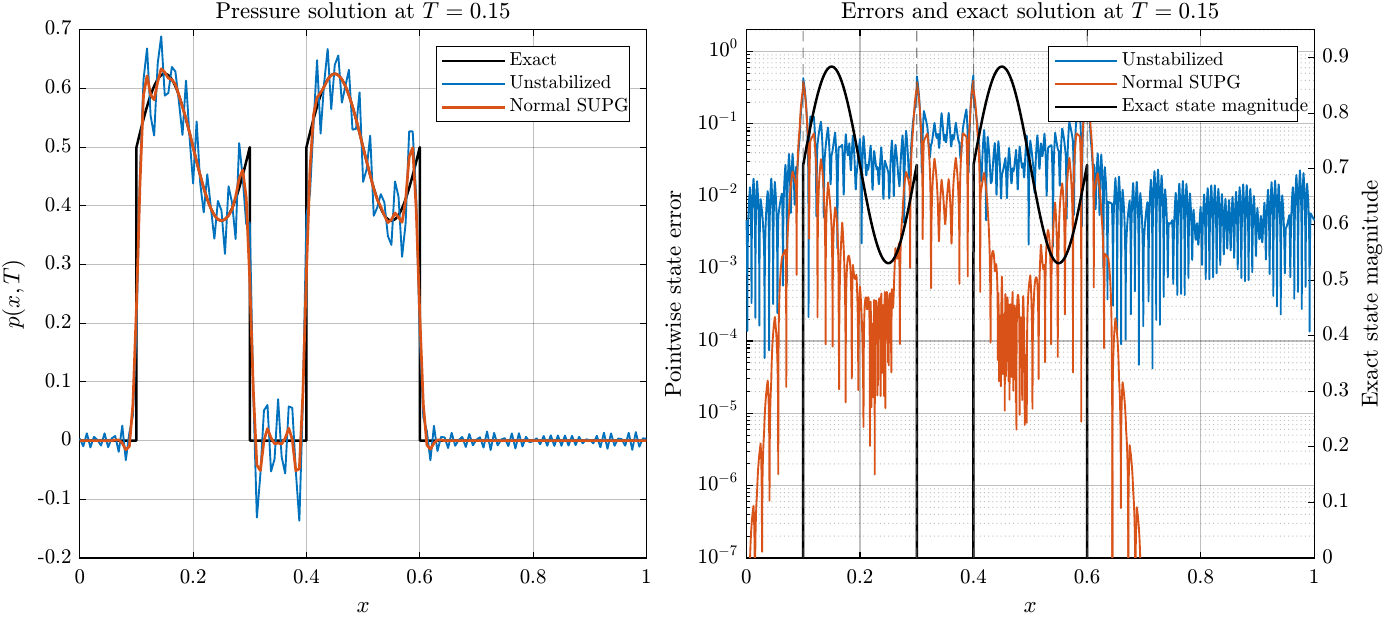}
\caption{Final pressure and pointwise state error on the \(160\)-element
mesh. In the right panel, the pointwise state errors are plotted against the
logarithmic left-hand axis, while the exact state magnitude is shown in black
against the right-hand axis. The unstabilised error is dispersive and global,
while normal-equation SUPG confines the significant error to neighbourhoods
of the exact jumps.}
\label{fig:numerical-discontinuous-profiles}
\end{figure}


\clearpage
\subsection{Two-Dimensional Localisation of a Compact Wave}
\label{subsec:numerical-localized-wave}
Finally, we ask whether stabilisation also suppresses numerical energy outside
the physical wave cone for smooth compact data in two dimensions. Let
\(x_0=(0.5,0.5)\), \(R_0=0.12\), and
\(r_{\mathrm{per}}(x)=\operatorname{dist}_{\mathrm{per}}(x,x_0)\). The
initial data are
\begin{equation}
p(x,0)=
\begin{cases}
\displaystyle \exp\!\left(1-\frac{1}{1-(r_{\mathrm{per}}(x)/R_0)^2}\right),
&r_{\mathrm{per}}(x)<R_0,\\
0,&r_{\mathrm{per}}(x)\geq R_0,
\end{cases}
\qquad u(x,0)=0
\label{eq:numerical-localized-initial-data}
\end{equation}
Thus the pressure is smooth, compactly supported, and equal to one at its
centre. We use \(64\) subdivisions in each coordinate direction and
Crank--Nicolson up to \(T=0.25\), with nominal step \(\tau_*=0.1h\). The two
runs use the same time grid and differ only in the stabilisation parameter:
\(\delta=0\) and the normal-equation value \(\delta=\tau/2\). With
\(h=\sqrt{2}/64\), define the exterior region by
\begin{equation}
\Omega_{\mathrm{ext}}(t)
=\bigl\{x\in\Omega:\operatorname{dist}_{\mathrm{per}}(x,x_0)
>R_0+t+2h\bigr\} \label{eq:numerical-localized-exterior-region}
\end{equation}
so the buffer outside the exact wave cone has width \(2h\). At final time the
cone-plus-buffer radius is \(0.414194<0.5\), and hence it has not reached a
periodic image.

For this experiment, let
\begin{align}
E(t)&=\frac12\int_{\Omega}\bigl(p_h^2+|u_h|^2\bigr)\,dx
\label{eq:numerical-localized-total-energy}\\
E_{\mathrm{ext}}(t)&=\frac12\int_{\Omega_{\mathrm{ext}}(t)}
\bigl(p_h^2+|u_h|^2\bigr)\,dx
\label{eq:numerical-localized-exterior-energy}
\end{align}

Figures~\ref{fig:numerical-localized-snapshots} and
\ref{fig:numerical-localized-energy} show the pressure fields on a common
colour scale and the corresponding energy histories. The unstabilised final
ratio \(E(T)/E(0)\) is \(1.0000\), whereas the normal-equation SUPG ratio is
\(0.993225\). Their final exterior-energy fractions
\(E_{\mathrm{ext}}(T)/E(0)\) are, respectively,
\(4.7772\times10^{-6}\) and \(1.2126\times10^{-7}\), so the stabilised
fraction is \(39.4\) times smaller. At \(T\), the full-state \(L^2\) difference
between the two solutions is \(3.1115\times10^{-3}\). Thus normal-equation
stabilisation improves localisation in this configuration, with modest
dissipation; no general multidimensional result is claimed.

\begin{figure}[htbp]
\centering
\includegraphics[width=\textwidth]{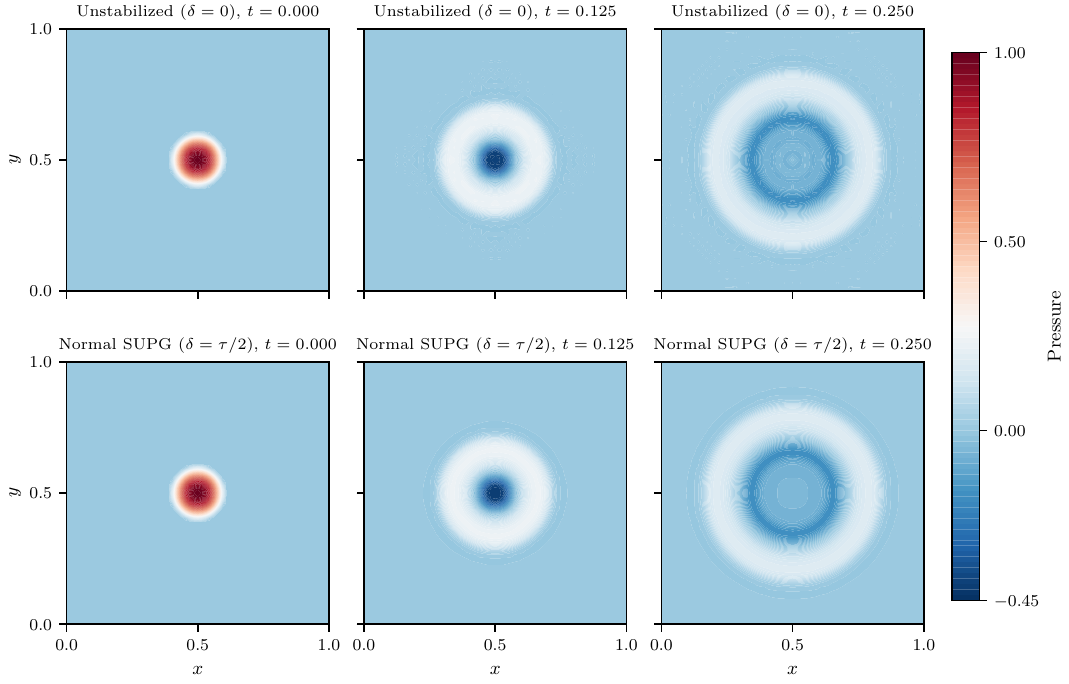}
\caption{Crank--Nicolson pressure on the periodic \(64\times64\) triangular
grid at \(t=0\), \(t=0.125\), and \(t=0.25\). The top row uses
\(\delta=0\), the bottom row uses the normal-equation value
\(\delta=\tau/2\), and all panels use the same colour scale.}
\label{fig:numerical-localized-snapshots}
\end{figure}

\begin{figure}[htbp]
\centering
\includegraphics[width=0.92\textwidth]{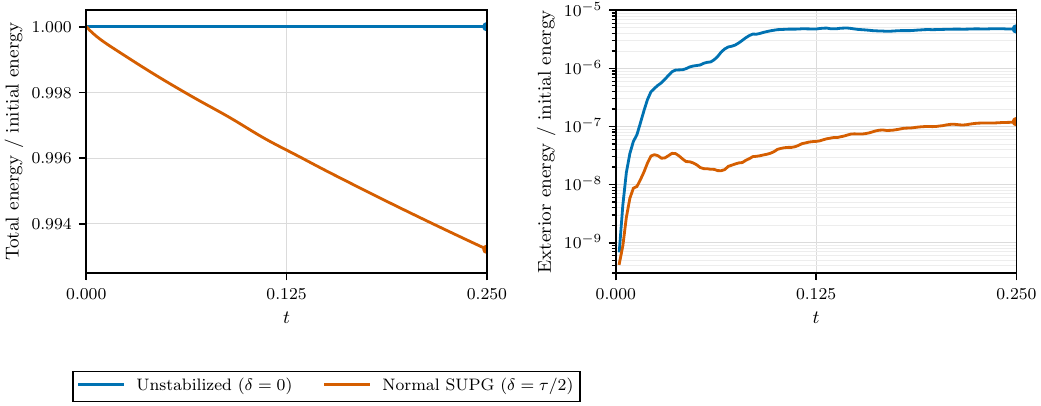}
\caption{Crank--Nicolson total-energy ratio and exterior-energy fraction for
\(\delta=0\) and \(\delta=\tau/2\) on the periodic \(64\times64\)
triangular grid over \(0\le t\le0.25\). Both quantities are normalised by
the initial total energy.}
\label{fig:numerical-localized-energy}
\end{figure}

\clearpage


\section{Conclusions}
\label{sec:conclusions}
The common operator and energy framework developed here leads to five
principal conclusions.

\begin{enumerate}[label=\arabic*.,leftmargin=*,itemsep=0.65em,topsep=0.5em]
\item \textbf{A common formulation clarifies the algebra.}
Separating the method average, the backward time difference, the material
residual, and the perturbed test operator places all the time integrators
considered here within the same structure. For implicit methods, the
distinguished choice \(\delta=b_0\tau\) turns each step into a symmetric
positive definite normal equation associated with residual minimisation in
the graph norm. A general SUPG choice remains admissible, but the resulting
system is not symmetric in general. For the explicit Adams--Bashforth
methods, \(b_0=0\), so no positive \(\delta\) can yield this normal-equation
structure.

\item \textbf{Stability depends on the method and the discrete space.}
The framework establishes stability for the \(\theta\)-method and AM3--AM5;
separate exact energy decompositions yield stability for AB3 and AB4. For
general polynomial degree, AM3 and AM4 require the strengthened time-step
restrictions in Section \ref{sec:adams-moulton}. For continuous piecewise
affine elements, cancellation due to the elementwise constancy of the
gradient restores the standard hyperbolic regime
\(\tau\lesssim\delta\lesssim h\). AM5, AB3, and AB4 are stable under standard
hyperbolic CFL restrictions. Since the general-degree Adams--Moulton analysis
imposes no lower bound on \(\delta\), it includes the normal-equation choice
whenever the relevant time-step restriction is satisfied.

\item \textbf{The dual residual yields a sharp fully discrete estimate.}
A direct dual-residual argument gives the bound
\(O(h^{k+1/2}+\tau^\nu)\) for \(\delta\sim h\) and \(\tau\lesssim h\), subject
to the stated stability and regularity assumptions. It uses the ordinary
\(L^2\)-projection, skew-symmetry, and discrete summation by parts.

\item \textbf{Order matching reveals the high-order regime.}
The standalone temporal tests confirm the formal order of each scheme. In the
principal one-dimensional acoustic experiment, pairing a method of temporal
order \(\nu\) with \(\mathbb P_{\nu-1}\) elements, \(\tau=0.1h\), and
\(\delta=b_0\tau\) produces convergence of orders two through five in the
final-time and graph errors. The material-residual errors follow the
half-order sequence \(3/2,5/2,7/2\), and \(9/2\). The two-dimensional
\(\mathbb P_1\) tests give graph and material-residual rates close to two and
\(3/2\), respectively.

\item \textbf{Stabilisation substantially improves localisation.}
In the discontinuous d'Alembert experiment, the unstabilised \(\mathbb P_1\) solution
converges locally at rate \(0.4995\), whereas normal-equation SUPG recovers the
optimal local rate \(2.0006\) away from the moving jumps. On the finest mesh,
the maximum local error of the stabilised method is smaller by a factor of
approximately \(1.45\times10^4\). In the smooth two-dimensional compact-wave
test, the stabilised solution has a final exterior-energy fraction that is
\(39.4\) times smaller, with only modest dissipation. Together, these
experiments provide evidence for the proposed localisation mechanism, but not
a general theorem for nonsmooth solutions.
\end{enumerate}

Natural theoretical extensions include variable coefficients and the
associated commutator terms, nonperiodic boundary conditions, and quantitative
local error estimates for nonsmooth solutions in multidimensional and
nonlinear systems. Computational priorities include robust multistep start-up
procedures and efficient solvers and preconditioners for both the symmetric
normal-equation systems and their generally nonsymmetric SUPG counterparts.

\clearpage

\appendix


\section{Proof of the Dual Residual Estimate}
\label{app:dual-residual-proof}
This appendix gives the complete proof of
Proposition~\ref{prop:dual-residual-bound}. The argument is based on the
ordinary \(L^2\)-projection, skew-symmetry, and discrete summation by parts.

Fix a terminal index \(m\), with \(n_0\le m\le N\), and extend the test sequence
by zero outside \(0\le n\le m\). Set
\begin{equation}
\mathcal D_m(V)
=\bigl(\mathcal I_{\tau}(V)+\|V\|_{\mathcal E,m}^2\bigr)^{1/2} \label{eq:dual-proof-denominator}
\end{equation}
The coercivity estimates \eqref{eq:energy-coercivity} and
\eqref{eq:dissipation-coercivity}, together with
\eqref{eq:test-operator-norm}, imply
\begin{align}
\max_{0\le j\le m}
\bigl(\|V^j\|_{\Omega}+\delta\|GV^j\|_{\Omega}\bigr)
&\lesssim\mathcal D_m(V) \label{eq:dual-proof-graph-control}\\
\|\delta^{1/2}A_{\tau}V\|_{\ell_{\tau}^2(n_0,m;\mathcal H)}
&\lesssim\mathcal D_m(V) \label{eq:dual-proof-residual-control}
\end{align}
For \(0\le j<n_0\), the first bound follows from the start-up energy
\(\mathcal I_{\tau}(V)\). Since the time stencil is fixed, we also have
\begin{equation}
\max_{n_0\le n\le m}\|\mathcal M_{\tau}V^n\|_{\Omega}
\lesssim\mathcal D_m(V) \label{eq:dual-proof-average-control}
\end{equation}

The spatial approximation estimate, combined with the finite-stencil trace
argument leading to \eqref{eq:method-average-discrete-time-trace}, gives
\begin{align}
\|\mathcal M_{\tau}\eta\|_{\ell_{\tau}^2(n_0,m;\mathcal H)}
&\lesssim h^r\mathcal N_{r,2}(U) \label{eq:dual-proof-projection-average}\\
\|A_{\tau}\eta\|_{\ell_{\tau}^2(n_0,m;\mathcal H)}
&\lesssim h^{r-1}\mathcal N_{r,2}(U) \label{eq:dual-proof-projection-residual}\\
\max_{n_0\le n\le m}\|A_{\tau}\eta^n\|_{\Omega}
+\|D_{\tau}^{+}A_{\tau}\eta\|_{
\ell_{\tau}^2(n_0,m-1;\mathcal H)}
&\lesssim h^{r-1}\mathcal N_{r,2}(U) \label{eq:dual-proof-differentiated-projection-residual}
\end{align}
Indeed, \(D_{\tau}^{-}\eta\) is an interval average of
\((\mathrm{Id}-\Pi_h)\partial_tU\), whereas
\(D_{\tau}^{+}D_{\tau}^{-}\eta\) is a difference of two such averages. The
remaining contributions are fixed-stencil averages of \(G\eta\) and
\(G D_{\tau}^{+}\eta\). Hence
\eqref{eq:dual-proof-projection-average}--
\eqref{eq:dual-proof-differentiated-projection-residual} follow from
\eqref{eq:spatial-approximation}, applied to \(U\) and its first two time
derivatives.


\subsection{Projection Contribution}
\label{subsec:dual-proof-projection}
Since \(\Pi_h\) is the spatial \(L^2\)-projection, it commutes with the time
operators and
\begin{align}
(D_{\tau}^{-}\eta^n,\mathcal M_{\tau}V^n)_{\Omega}
&=0 \label{eq:dual-proof-first-orthogonality}\\
(\mathcal M_{\tau}\eta^n,D_{\tau}^{-}V^n)_{\Omega}
&=0 \label{eq:dual-proof-second-orthogonality}
\end{align}
Skew-symmetry of \(G\) therefore gives
\begin{equation}
(A_{\tau}\eta^n,\mathcal M_{\tau}V^n)_{\Omega}
=-(\mathcal M_{\tau}\eta^n,A_{\tau}V^n)_{\Omega} \label{eq:dual-proof-projection-cancellation}
\end{equation}
Combining \eqref{eq:dual-proof-projection-cancellation} with
\(P_{\delta}=\mathcal M_{\tau}+\delta A_{\tau}
+\delta^2GD_{\tau}^{-}\) yields
\eqref{eq:dual-projection-functional-identity}. Notice that this step does not
move \(\mathcal M_{\tau}\) between the two arguments.

For the projection contribution, define
\begin{equation}
\mathcal F_{\eta,m}(V)
=\tau\sum_{n=n_0}^{m}
(A_{\tau}\eta^n,P_{\delta}V^n)_{\Omega} \label{eq:dual-proof-projection-functional}
\end{equation}
Since \(GD_{\tau}^{-}V=D_{\tau}^{-}GV\), the exact summation-by-parts
formula yields
\begin{align}
\mathcal F_{\eta,m}(V)
&=\tau\sum_{n=n_0}^{m}
(\delta A_{\tau}\eta^n-\mathcal M_{\tau}\eta^n,
A_{\tau}V^n)_{\Omega} \label{eq:dual-proof-projection-functional-main}\\
&\quad+\delta^2(A_{\tau}\eta^m,GV^m)_{\Omega}
-\delta^2(A_{\tau}\eta^{n_0},GV^{n_0-1})_{\Omega} \notag\\
&\quad-\tau\delta^2\sum_{n=n_0}^{m-1}
(D_{\tau}^{+}A_{\tau}\eta^n,GV^n)_{\Omega} \label{eq:dual-proof-projection-functional-sbp}
\end{align}
Both endpoint terms are retained. The Cauchy--Schwarz inequality and
\eqref{eq:dual-proof-graph-control}--\eqref{eq:dual-proof-residual-control}
then yield, with \(T_m=m\tau\le T\),
\begin{align}
|\mathcal F_{\eta,m}(V)|
&\lesssim\Bigl(
\delta^{1/2}\|A_{\tau}\eta\|_{\ell_{\tau}^2(n_0,m;\mathcal H)}
+\delta^{-1/2}\|\mathcal M_{\tau}\eta\|_{
\ell_{\tau}^2(n_0,m;\mathcal H)}\Bigr)\mathcal D_m(V) \label{eq:dual-proof-projection-functional-bound-main}\\
&\quad+\delta\Bigl(
2\max_{n_0\le n\le m}\|A_{\tau}\eta^n\|_{\Omega}
+T_m^{1/2}\|D_{\tau}^{+}A_{\tau}\eta\|_{
\ell_{\tau}^2(n_0,m-1;\mathcal H)}\Bigr)\mathcal D_m(V) \label{eq:dual-proof-projection-functional-bound-sbp}
\end{align}
Substituting \eqref{eq:dual-proof-projection-average}--
\eqref{eq:dual-proof-differentiated-projection-residual}, together with
\(0<\delta\le1\), gives
\begin{equation}
|\mathcal F_{\eta,m}(V)|
\lesssim\bigl((1+T^{1/2})\delta^{1/2}h^{r-1}
+\delta^{-1/2}h^r\bigr)\mathcal N_{r,2}(U)\mathcal D_m(V) \label{eq:dual-proof-projection-final}
\end{equation}


\subsection{Truncation Contribution}
\label{subsec:dual-proof-truncation}
For the truncation contribution, set
\begin{equation}
\mathcal F_{\rho,m}(V)
=\tau\sum_{n=n_0}^{m}(\rho^n,P_{\delta}V^n)_{\Omega} \label{eq:dual-proof-truncation-functional}
\end{equation}
Expanding \(P_{\delta}\) and applying the same summation-by-parts formula give
\begin{align}
\mathcal F_{\rho,m}(V)
&=\tau\sum_{n=n_0}^{m}(\rho^n,\mathcal M_{\tau}V^n)_{\Omega} \label{eq:dual-proof-truncation-functional-main}\\
&\quad+\tau\delta\sum_{n=n_0}^{m}(\rho^n,A_{\tau}V^n)_{\Omega} \notag\\
&\quad+\delta^2(\rho^m,GV^m)_{\Omega}
-\delta^2(\rho^{n_0},GV^{n_0-1})_{\Omega} \notag\\
&\quad-\tau\delta^2\sum_{n=n_0}^{m-1}
(D_{\tau}^{+}\rho^n,GV^n)_{\Omega} \label{eq:dual-proof-truncation-functional-sbp}
\end{align}
Consequently, the bounds
\eqref{eq:dual-proof-graph-control}--\eqref{eq:dual-proof-average-control} and
the Cauchy--Schwarz inequality give
\begin{align}
|\mathcal F_{\rho,m}(V)|
&\lesssim\Bigl(
T_m^{1/2}\|\rho\|_{\ell_{\tau}^2(n_0,m;\mathcal H)}
+\delta^{1/2}\|\rho\|_{\ell_{\tau}^2(n_0,m;\mathcal H)}\Bigr)
\mathcal D_m(V) \label{eq:dual-proof-truncation-functional-bound-main}\\
&\quad+\delta\Bigl(
2\max_{n_0\le n\le m}\|\rho^n\|_{\Omega}
+T_m^{1/2}\|D_{\tau}^{+}\rho\|_{
\ell_{\tau}^2(n_0,m-1;\mathcal H)}\Bigr)\mathcal D_m(V) \label{eq:dual-proof-truncation-functional-bound-sbp}
\end{align}
The consistency estimates \eqref{eq:global-truncation-bound}--
\eqref{eq:differentiated-truncation-bound}, together with
\(0<\delta\le1\), therefore imply
\begin{equation}
|\mathcal F_{\rho,m}(V)|
\lesssim(1+T^{1/2})\tau^\nu
\mathcal C^{\sharp}_{T,\nu+2}(U,F)\mathcal D_m(V) \label{eq:dual-proof-truncation-final}
\end{equation}

Finally, since \(R=A_{\tau}\eta-\rho\), its residual functional is
\(\mathcal F_{\eta,m}-\mathcal F_{\rho,m}\). Combining
\eqref{eq:dual-proof-projection-final} and
\eqref{eq:dual-proof-truncation-final}, dividing by
\(\mathcal D_m(V)\), taking the supremum over all admissible test sequences,
and then maximising over \(n_0\le m\le N\) proves
\eqref{eq:dual-residual-bound-space}--\eqref{eq:dual-residual-bound-time}.

\bigskip

\paragraph{Acknowledgements.}
The authors worked on this project in autumn 2025 while visiting Institut
Mittag-Leffler for the semester programme ``Interfaces and Unfitted
Discretization Methods''. We gratefully acknowledge the support and hospitality
of Institut Mittag-Leffler.

EB was supported in part by EPSRC grants EP/X042650/1 and EP/V050400/1. PH was
supported by the Swedish Research Council under grant 2022-03908. MGL was
supported in part by the Swedish Research Council under grants 2021-04925 and
2025-05562, the Knut and Alice Wallenberg Foundation under grant KAW
2025.0277, and the Swedish Research Programme Essence.

\paragraph{Use of Artificial Intelligence and Computational Tools.}
During the preparation and revision of this manuscript, the authors used
OpenAI's ChatGPT and Codex for language editing, organisation of the
presentation, LaTeX drafting, and algebraic checks of the Adams stability
identities. MATLAB and Python, using NumPy, SciPy, and Matplotlib, were used for
numerical computations, independent verification, and figure preparation. All
AI-assisted material was reviewed and verified by the authors, who take full
responsibility for the content of the manuscript.

\paragraph{Code Availability.}
The numerical code and reference data used for the experiments in this paper
are publicly available at
\url{github.com/mglarson1/symtransport-numerics}. The repository
contains MATLAB and Python implementations, fixed experiment configurations,
scripts for reproducing the reported tables and figures, reference results,
and automated verification tests. The software is distributed under the BSD
3-Clause License.

\bigskip
\bigskip
\noindent
\footnotesize {\bf Authors' addresses:}

\smallskip
\noindent
Erik Burman, \quad \hfill \addressuclshort\\
{\tt e.burman@ucl.ac.uk}

\smallskip
\noindent
Peter Hansbo, \quad \hfill \addressjushort\\
{\tt peter.hansbo@ju.se}

\smallskip
\noindent
Mats G. Larson, \quad \hfill \addressumushort\\
{\tt mats.larson@umu.se}

\bibliographystyle{abbrv}
\bibliography{references}

\begin{thebibliography}{10}

\bibitem{AKS19}
M.~Asadzadeh, P.~Kowalczyk, and C.~Standar.
\newblock On {$hp$}-streamline diffusion and {N}itsche schemes for the
  relativistic {V}lasov--{M}axwell system.
\newblock {\em Kinet. Relat. Models}, 12(1):105--131, 2019.

\bibitem{BB04}
R.~Becker and M.~Braack.
\newblock A two-level stabilization scheme for the {N}avier-{S}tokes equations.
\newblock In {\em Numerical mathematics and advanced applications}, pages
  123--130. Springer, Berlin, 2004.

\bibitem{BDG25}
L.~Beir{\~a}o~da Veiga, F.~Dassi, and S.~G{\'o}mez.
\newblock {SUPG}-stabilized time-{DG} finite and virtual elements for the
  time-dependent advection-diffusion equation.
\newblock {\em Comput. Methods Appl. Mech. Engrg.}, 436:117722, 2025.

\bibitem{BH21}
M.~Berggren and L.~H\"{a}gg.
\newblock Well-posed variational formulations of {F}riedrichs-type systems.
\newblock {\em J. Differential Equations}, 292:90--131, 2021.

\bibitem{BM04}
O.~Besson and G.~de~Montmollin.
\newblock Space-time integrated least squares: a time-marching approach.
\newblock {\em Internat. J. Numer. Methods Fluids}, 44(5):525--543, 2004.

\bibitem{BGS04}
P.~B. Bochev, M.~D. Gunzburger, and J.~N. Shadid.
\newblock Stability of the {SUPG} finite element method for transient
  advection-diffusion problems.
\newblock {\em Comput. Methods Appl. Mech. Engrg.}, 193(23-26):2301--2323,
  2004.

\bibitem{BrooksHughes1982CMAME}
A.~N. Brooks and T.~J.~R. Hughes.
\newblock Streamline upwind/{P}etrov--{G}alerkin formulations for convection
  dominated flows with particular emphasis on the incompressible
  {N}avier--{S}tokes equations.
\newblock {\em Computer Methods in Applied Mechanics and Engineering},
  32(1--3):199--259, 1982.

\bibitem{BE16}
K.~Burazin and M.~Erceg.
\newblock Non-stationary abstract {F}riedrichs systems.
\newblock {\em Mediterr. J. Math.}, 13(6):3777--3796, 2016.

\bibitem{Bu10}
E.~Burman.
\newblock Consistent {SUPG}-method for transient transport problems: stability
  and convergence.
\newblock {\em Comput. Methods Appl. Mech. Engrg.}, 199(17-20):1114--1123,
  2010.

\bibitem{BEF10}
E.~Burman, A.~Ern, and M.~A. Fern\'{a}ndez.
\newblock Explicit {R}unge-{K}utta schemes and finite elements with symmetric
  stabilization for first-order linear {PDE} systems.
\newblock {\em SIAM J. Numer. Anal.}, 48(6):2019--2042, 2010.

\bibitem{BG22}
E.~Burman and J.~Guzm\'an.
\newblock Implicit-explicit multistep formulations for finite element
  discretisations using continuous interior penalty.
\newblock {\em ESAIM Math. Model. Numer. Anal.}, 56(1):349--383, 2022.

\bibitem{BH04}
E.~Burman and P.~Hansbo.
\newblock Edge stabilization for {G}alerkin approximations of
  convection-diffusion-reaction problems.
\newblock {\em Comput. Methods Appl. Mech. Engrg.}, 193(15-16):1437--1453,
  2004.

\bibitem{BS11}
E.~Burman and G.~Smith.
\newblock Analysis of the space semi-discretized {SUPG} method for transient
  convection-diffusion equations.
\newblock {\em Math. Models Methods Appl. Sci.}, 21(10):2049--2068, 2011.

\bibitem{CS89}
B.~Cockburn and C.-W. Shu.
\newblock T{VB} {R}unge-{K}utta local projection discontinuous {G}alerkin
  finite element method for conservation laws. {II}. {G}eneral framework.
\newblock {\em Math. Comp.}, 52(186):411--435, 1989.

\bibitem{Cod00}
R.~Codina.
\newblock Stabilization of incompressibility and convection through orthogonal
  sub-scales in finite element methods.
\newblock {\em Comput. Methods Appl. Mech. Engrg.}, 190(13-14):1579--1599,
  2000.

\bibitem{Dahl63}
G.~G. Dahlquist.
\newblock A special stability problem for linear multistep methods.
\newblock {\em Nordisk Tidskr. Informationsbehandling (BIT)}, 3:27--43, 1963.

\bibitem{Der12}
E.~Deriaz.
\newblock Stability conditions for the numerical solution of
  convection-dominated problems with skew-symmetric discretizations.
\newblock {\em SIAM J. Numer. Anal.}, 50(3):1058--1085, 2012.

\bibitem{EG06}
A.~Ern and J.-L. Guermond.
\newblock Discontinuous {G}alerkin methods for {F}riedrichs' systems. {I}.
  {G}eneral theory.
\newblock {\em SIAM J. Numer. Anal.}, 44(2):753--778, 2006.

\bibitem{EG25}
A.~Ern and J.-L. Guermond.
\newblock {$L^2$}-stability of explicit {Runge--Kutta} methods with {SUPG}
  stabilization for transient transport problems.
\newblock Research report hal-05363804, HAL, 2025.
\newblock Available at \url{https://hal.science/hal-05363804v1}.

\bibitem{EGC07}
A.~Ern, J.-L. Guermond, and G.~Caplain.
\newblock An intrinsic criterion for the bijectivity of {H}ilbert operators
  related to {F}riedrichs' systems.
\newblock {\em Comm. Partial Differential Equations}, 32(1-3):317--341, 2007.

\bibitem{FGK25}
T.~Führer, R.~González, and M.~Karkulik.
\newblock Well-posedness of first-order acoustic wave equations and space-time
  finite element approximation.
\newblock {\em IMA Journal of Numerical Analysis}, page drae104, 03 2025.

\bibitem{GFR15}
M.~L. Ghrist, B.~Fornberg, and J.~A. Reeger.
\newblock Stability ordinates of {A}dams predictor-corrector methods.
\newblock {\em BIT}, 55(3):733--750, 2015.

\bibitem{Guer99}
J.-L. Guermond.
\newblock Stabilization of {G}alerkin approximations of transport equations by
  subgrid modeling.
\newblock {\em M2AN Math. Model. Numer. Anal.}, 33(6):1293--1316, 1999.

\bibitem{HFH89}
T.~J.~R. Hughes, L.~P. Franca, and G.~M. Hulbert.
\newblock A new finite element formulation for computational fluid dynamics.
  {VIII}. {T}he {G}alerkin/least-squares method for advective-diffusive
  equations.
\newblock {\em Comput. Methods Appl. Mech. Engrg.}, 73(2):173--189, 1989.

\bibitem{JNP84}
C.~Johnson, U.~N\"{a}vert, and J.~Pitk\"{a}ranta.
\newblock Finite element methods for linear hyperbolic problems.
\newblock {\em Comput. Methods Appl. Mech. Engrg.}, 45(1-3):285--312, 1984.

\bibitem{LW95}
G.~Lube and D.~Weiss.
\newblock Stabilized finite element methods for singularly perturbed parabolic
  problems.
\newblock {\em Appl. Numer. Math.}, 17(4):431--459, 1995.

\bibitem{MTRA23}
S.~Michel, D.~Torlo, M.~Ricchiuto, and R.~Abgrall.
\newblock Spectral analysis of high order continuous {FEM} for hyperbolic
  {PDE}s on triangular meshes: influence of approximation, stabilization, and
  time-stepping.
\newblock {\em J. Sci. Comput.}, 94:49, 2023.

\bibitem{ZL22}
J.~Zhang and X.~Liu.
\newblock Uniform stability of the {SUPG} method for the evolutionary
  convection-diffusion-reaction equation.
\newblock {\em Comput. Math. Appl.}, 124:1--6, 2022.

\end{thebibliography}

\end{document}